\documentclass[reqno]{amsart}
\usepackage{amsmath, amsthm, amssymb, amscd}

\usepackage[a4paper,width=160mm,top=25mm,bottom=25mm,centering]{geometry}

\usepackage{enumitem}
\usepackage{graphicx} 
\usepackage[dvipsnames]{xcolor}
\usepackage{comment}
\usepackage{cases}
\usepackage{array}
\usepackage{multirow}
\usepackage{longtable}
\usepackage{microtype}

\makeatletter
\newsavebox{\@brx}
\newcommand{\llangle}[1][]{\savebox{\@brx}{\(\m@th{#1\langle}\)}%
  \mathopen{\copy\@brx\kern-0.5\wd\@brx\usebox{\@brx}}}
\newcommand{\rrangle}[1][]{\savebox{\@brx}{\(\m@th{#1\rangle}\)}%
  \mathclose{\copy\@brx\kern-0.5\wd\@brx\usebox{\@brx}}}
\makeatother

\usepackage{tikz}
\usepackage{pgf}
\usepackage{tikzscale}
\usepackage{tikz-cd}
\usetikzlibrary{decorations.pathreplacing,decorations.markings}
\usetikzlibrary{arrows, knots, external, positioning}

\makeatletter
\tikzset{column sep/.code=\def\pgfmatrixcolumnsep{\pgf@matrix@xscale*(#1)},
  row sep/.code   =\def\pgfmatrixrowsep{\pgf@matrix@yscale*(#1)},
  matrix xscale/.code=%
    \pgfmathsetmacro\pgf@matrix@xscale{\pgf@matrix@xscale*(#1)},
  matrix yscale/.code=%
    \pgfmathsetmacro\pgf@matrix@yscale{\pgf@matrix@yscale*(#1)},
  matrix scale/.style={/tikz/matrix xscale={#1},/tikz/matrix yscale={#1}}}
\def\pgf@matrix@xscale{1}
\def\pgf@matrix@yscale{1}
\makeatother

\usepackage{url}
\usepackage{caption} 

\newtheorem{theorem}{Theorem}[section]

\newtheorem{lemma}[theorem]{Lemma}
\newtheorem{proposition}[theorem]{Proposition}
\newtheorem{corollary}[theorem]{Corollary}

\theoremstyle{definition}

\newtheorem{remark}[theorem]{Remark}

\newcommand{\CPbar}{\overline{\mathbb{CP}}{}^2}

\numberwithin{equation}{section}

\usepackage[bookmarksnumbered=true, colorlinks=true]{hyperref}
\AtBeginDocument{\hypersetup{citecolor=BlueGreen, linkcolor=BlueGreen, urlcolor=BlueGreen}}

\begin{document}

\title[Exotic involutions on irreducible 4-manifolds from Real Seiberg--Witten theory]{Exotic involutions on irreducible 4-manifolds from Real Seiberg--Witten theory}

\author{Woohyeok Jo}
\address{HUN-REN R\'enyi Institute of Mathematics, H-1053 Budapest, Re\'altanoda utca 13--15, Hungary}
\email{jo.woohyeok@renyi.hu}

\author{Simone Tagliente}
\address{HUN-REN R\'enyi Institute of Mathematics, H-1053 Budapest, Re\'altanoda utca 13--15, Hungary and Eötvös Loránd University, Pázmány Péter sétány 1/C., H-1117 Budapest, Hungary}
\email{tagliente.simone@renyi.hu}

\date\today

\begin{abstract}
Using Real Seiberg--Witten theory, Miyazawa constructed infinitely many exotic involutions on $\CPbar$, and Baraglia extended this construction to $X\#\CPbar$ for a broad class of $4$-manifolds $X$ equipped with an involution $c$. We show that, under suitable Real Seiberg--Witten hypotheses, the $\CPbar$ summand can be omitted when $X$ is non-spin and the complement of the fixed locus is simply-connected. Our construction also yields infinite families of topologically isotopic but smoothly non-isotopic embeddings of the branch surface in the quotient $X/\langle c\rangle$. As applications, we construct anti-holomorphic involutions on irreducible, simply-connected projective surfaces of general type with infinitely many exotic copies, none of which is anti-symplectic with respect to any orientation-compatible symplectic form. These surfaces include cyclic branched covers of $\mathbb{CP}^2$ and complete intersections of multidegree $(2,4k)$ in $\mathbb{CP}^4$.
\end{abstract}

\maketitle

\section{Introduction}

For an oriented smooth $4$-manifold $X$, two orientation-preserving smooth involutions $c,c'$ on $X$ are said to be \emph{topologically} (resp. \emph{smoothly}) \emph{equivalent} if there exists an orientation-preserving homeomorphism (resp. diffeomorphism) $f:X\to X$ with $f\circ c=c'\circ f$. A family of orientation-preserving smooth involutions on $X$ that are all topologically equivalent but pairwise smoothly inequivalent is called an \emph{exotic family}. Exotic smooth involutions have been constructed by several methods; see \cite{Ue-1996, Ue-1998, Hanke-Kotschick-Wehrheim-2003, Fintushel-Stern-Sunukjian-2009, Kwasik-Schultz-2017}. 

Recently, Miyazawa's construction \cite{Miyazawa-2023} using Real Seiberg--Witten theory, together with a diffeomorphism result of Hughes--Kim--Miller, produced an infinite exotic family of involutions on $\CPbar$ containing complex conjugation; see \cite[Corollary 1.5]{Hughes-Kim-Miller-2024}. Baraglia \cite{Baraglia-2026} then extended this result to show that $X\#\CPbar$ has an infinite exotic family of involutions for a broad class of $4$-manifolds $X$ with an involution; see Theorem \ref{thm:Baraglia} below. The first main result of this paper is to construct an infinite exotic family of involutions on $X$ itself, without taking a connected sum with an additional $\CPbar$, under the hypotheses below. 

\begin{theorem}\label{thm:main}
    Let $X$ be a closed, oriented, smooth, non-spin $4$-manifold with an orientation-preserving smooth involution $c:X\to X$ such that the fixed locus $X^c$ is a nonempty surface and $X-X^c$ is simply-connected. Assume that $(X,c)$ admits a Real spin$^c$ structure $\mathfrak{s}$ satisfying \[
    b_+(X)^{-c}=\frac{c_1(\mathfrak{s})^2-\sigma(X)}{8} \quad \text{and} \quad \deg_R(X,c,\mathfrak{s})\neq 0.
    \]
    Then there exists an infinite exotic family $\{c_n\}_{n=0}^\infty$ of involutions on $X$ with $c_0=c$, satisfying the following additional properties. \begin{itemize}[leftmargin=2em]
        \item For every $n\geq 1$, there is no symplectic form $\omega$ on $X$ inducing the given orientation such that $c_n^*\omega=-\omega$. That is, $c_n$ is not anti-symplectic with respect to any orientation-compatible symplectic form on $X$.
        \item The family remains exotic after any number of equivariant blow-ups at fixed points. More precisely, if $\kappa$ denotes complex conjugation on $\CPbar$, then $\{c_n\# k\kappa\}_{n=0}^\infty$ is an exotic family of involutions on $X\#k \CPbar$ for every $k\geq 0$. 
    \end{itemize}
\end{theorem}

Here, a Real spin$^c$ structure consists of a spin$^c$ structure together with a compatible anti-linear lift of $c$ to the spinor bundle. We write $b_+(X)^{-c}$ for the dimension of the $(-1)$-eigenspace of $c^*$ on $H^2_+(X;\mathbb{R})$, and $\deg_R(X,c,\mathfrak{s})$ denotes the associated Real degree; see Section \ref{sec:RSW} for details. 

\vspace{1mm}
Our construction has a corresponding consequence for embeddings of surfaces. For a smooth $4$-manifold $Q$, two smoothly embedded closed surfaces $B,B'\subset Q$ are said to be \emph{topologically} (resp. \emph{smoothly}) \emph{isotopic} if there exists a homeomorphism (resp. diffeomorphism) $f:Q\to Q$ that is topologically (resp. smoothly) isotopic to the identity such that $f(B)=B'$. A family of smoothly embedded surfaces in $Q$ that are all topologically isotopic but pairwise smoothly non-isotopic is called an \emph{exotic family}. Exotic surfaces have been studied extensively; for example, see \cite{Finashin-Kreck-Viro-1988, Fintushel-Stern-1997, Finashin-2002, Kim-2006, Finashin-2009, Mark-2013, Auckly-2019, Hoffman-Sunukjian-2020, Levine-Lidman-Piccirillo-2023, Baraglia-2024}. 

Notably, Miyazawa \cite{Miyazawa-2023} constructed an exotic family $\{P_n\}_{n=0}^\infty$ of surfaces in $S^4$, where $P_0$ is the standard $\mathbb{RP}^2$ with normal Euler number $2$. Baraglia \cite{Baraglia-2026} extended this result to show that, for a certain class of $4$-manifold-involution pairs $(X,c)$ with connected fixed loci $X^c$, there is an infinite exotic family of surfaces in the quotient $Q:=X/\langle c\rangle$ containing $B\# P_0$, where $B$ is the image of $X^c$ in $Q$ (see Theorem \ref{thm:Baraglia}). Our second result shows that, under the hypotheses of Theorem \ref{thm:main}, there is an infinite exotic family of surfaces in $Q$ containing $B$ itself, without taking a connected sum with $P_0$ or assuming $X^c$ is connected. 

\begin{theorem}\label{thm:isotopy}
    Let $(X,c)$ satisfy the hypotheses of Theorem \ref{thm:main}, and let $B$ be the image of the fixed locus $X^c$ in $Q:=X/\langle c\rangle$. Then there exists an infinite exotic family $\{B_n\}_{n=0}^\infty$ of surfaces in $Q$ with $B_0=B$. 
\end{theorem}

Our proof actually shows the stronger result that the pairs $(Q,B_n)$ and $(Q,B_m)$ are not diffeomorphic for any $n\neq m$. 

The hypotheses of Theorem \ref{thm:main} can be verified in terms of the ordinary Seiberg--Witten invariant.
\begin{corollary}\label{cor:ordinary_SW_formulation}
    Let $X$ be a closed, oriented, non-spin, smooth $4$-manifold with an orientation-preserving smooth involution $c:X\to X$ such that the fixed locus $X^c$ is a nonempty surface and $X-X^c$ is simply-connected. Suppose that $b_+(X)\equiv 3\pmod{4}$ and that $X$ admits a spin$^c$ structure $\mathfrak{s}$ such that $c^*\mathfrak{s}\cong \overline{\mathfrak{s}}$, $\operatorname{SW}(X,\mathfrak{s})\equiv 1\pmod{2}$, and \[
    \frac{c_1(\mathfrak{s})^2-\sigma(X)}{8}=\frac{b_+(X)+1}{2}=b_+(X)^{-c}.
    \]
    Then the conclusions of Theorems \ref{thm:main} and \ref{thm:isotopy} hold.
\end{corollary}

Applying this criterion to the canonical spin$^c$ structure of an anti-symplectic involution gives the following corollary. 
\begin{corollary}\label{cor:symplectic}
    Let $(X,\omega)$ be a closed, non-spin, symplectic $4$-manifold with $b_+(X)\equiv 3\pmod{4}$, and let $c:X\to X$ be a smooth involution with $c^*\omega=-\omega$ such that the fixed locus $X^c$ is a nonempty surface and $X-X^c$ is simply-connected. Then the conclusions of Theorems \ref{thm:main} and \ref{thm:isotopy} hold, for the orientation on $X$ induced by $\omega$.
\end{corollary}

Recall that a smooth $4$-manifold $X$ is \emph{irreducible} if for every smooth connected sum decomposition $X=X_1\#X_2$, either $X_1$ or $X_2$ is a homotopy $4$-sphere. Our construction can produce an infinite exotic family of involutions on simply-connected, irreducible $4$-manifolds. For integers $r\geq 2$ and $m\geq 1$, consider the cyclic $r$-fold cover of $\mathbb{CP}^2$ branched along a nonsingular curve of degree $rm$, equipped with its complex orientation. Its oriented diffeomorphism type depends only on $r$ and $m$, and its holomorphic Euler characteristic is given by \[
\chi_{r,m}:=\sum_{k=0}^{r-1}\frac{(km-1)(km-2)}{2};
\]
see Section \ref{sec:branched_covers} for details. 
\begin{theorem}\label{thm:branched_cover}
   If $m$ and $\chi_{r,m}$ are both even, then there exists a simply-connected, irreducible complex surface $X_{r,m}$ of general type given by the cyclic $r$-fold cover of $\mathbb{CP}^2$ branched along a nonsingular curve of degree $rm$, admitting a K\"ahler form $\omega$ and an anti-holomorphic involution $c:X_{r,m}\to X_{r,m}$ such that the fixed locus is \[
        X_{r,m}^c\cong \begin{cases}
            \mathbb{RP}^2 & \text{if $r$ is odd}, \\
            \mathbb{RP}^2\amalg \mathbb{RP}^2 & \text{if $r$ is even},
        \end{cases}
        \]
    and $(X_{r,m},\omega,c)$ satisfies the hypotheses of Corollary \ref{cor:symplectic}. Consequently, the following statements hold. \begin{itemize}[leftmargin=2em]
       \item $X_{r,m}$ admits an infinite exotic family $\{c_n\}_{n=0}^\infty$ of involutions with $c_0=c$, such that every $c_n$ with $n\geq 1$ is not anti-symplectic with respect to any orientation-compatible symplectic form.
       \item There exists an infinite exotic family $\{B_{n}\}_{n=0}^\infty$ of surfaces in the quotient $Q_{r,m}:=X_{r,m}/\langle c\rangle$ such that $B_{0}$ is the image of the fixed locus $X_{r,m}^c$ in $Q_{r,m}$.
   \end{itemize}
\end{theorem}

A second family of examples is provided by smooth complete intersection surfaces in $\mathbb{CP}^4$. Recall that a surface in $\mathbb{CP}^4$ is said to be a \emph{complete intersection} of multidegree $(d,d')$ if it is given by $\{p=p'=0\}$ for some homogeneous polynomials $p$ and $p'$ of degrees $d$ and $d'$, respectively, such that their gradients are linearly independent at every nonzero common zero in $\mathbb{C}^5$. Its oriented diffeomorphism type depends only on $d$ and $d'$.

\begin{theorem}\label{thm:complete_intersection}
For each integer $k\geq 1$, there exists a simply-connected, irreducible complex surface $Y_k$ of general type given by a complete intersection in $\mathbb{CP}^4$ of multidegree $(2,4k)$, admitting a K\"ahler form $\omega$ and an anti-holomorphic involution $c:Y_k\to Y_k$ such that the fixed locus $Y_k^c$ is a $2$-sphere and $(Y_k,\omega,c)$ satisfies the hypotheses of Corollary \ref{cor:symplectic}. Consequently, the following statements hold. \begin{itemize}[leftmargin=2em]
    \item $Y_k$ admits an infinite exotic family $\{c_n\}_{n=0}^\infty$ of involutions with $c_0=c$, such that every $c_n$ with $n\geq 1$ is not anti-symplectic with respect to any orientation-compatible symplectic form. 
    \item There is an infinite exotic family $\{B_{n}\}_{n=0}^\infty$ of $2$-spheres in the quotient $Q_k:=Y_k/\langle c\rangle$ such that $B_{0}$ is the image of the fixed locus $Y_k^c$ in $Q_k$. 
\end{itemize}
\end{theorem}

\begin{remark}
(1) In both Theorems \ref{thm:branched_cover} and \ref{thm:complete_intersection}, although each $c_n$ for $n\geq 1$ is not anti-symplectic with respect to any orientation-compatible symplectic form, we have $c_n^*[\omega]=-[\omega]$ on cohomology for the K\"ahler form $\omega$; see Remarks \ref{rmk:cohomology_antisymplectic} and \ref{rmk:complete_antisymplectic}.

(2) In Theorem \ref{thm:branched_cover}, the condition that both $m$ and $\chi_{r,m}$ are even is equivalent to either $m\equiv 0\pmod{4}$ and $r\equiv 0 \pmod{2}$, or $m\equiv 2\pmod{4}$ and $r\equiv 0,3\pmod {4}$; see Remark \ref{rmk:chi_r,m}. For example, Theorem \ref{thm:branched_cover} applies to the double cover $X_{2,4k}$ of $\mathbb{CP}^2$ branched along a nonsingular curve of degree $8k$.

(3) The oriented homeomorphism types of $X_{r,m}/\langle c\rangle$ and $Y_k/\langle c\rangle$ are given in Propositions \ref{prop:homeomorphism_type} and \ref{prop:quotient_homeo_type_complete_intersection}, respectively.
\end{remark}

Next, we consider equivariant connected sums. For comparison, we first recall Baraglia's results. For a closed, oriented, smooth $4$-manifold $X$ with an orientation-preserving smooth involution $c:X\to X$, Baraglia \cite[Definition 1.1]{Baraglia-2026} defined the notion of \emph{admissibility} for the pair $(X,c)$; see Section \ref{sec:main_proofs} for the precise definition. Baraglia proved that an equivariant connected sum of admissible pairs admits an infinite exotic family after one additional equivariant blow-up:
\begin{theorem}[{\cite[Theorem 1.2]{Baraglia-2026}}]\label{thm:Baraglia}
    Let $(X_1,c_1),\dots,(X_k,c_k)$ be admissible pairs, and let $\kappa$ denote complex conjugation on $\CPbar$. Form the equivariant connected sum \[
    (X,c):=(X_1\# \cdots \# X_k, c_1\# \cdots \# c_k)
    \]
    at fixed points. Then $X\# \CPbar$ admits an infinite exotic family of involutions containing $c\# \kappa$. Furthermore, if every fixed locus $X_i^{c_i}$ is connected, then there exists an infinite exotic family of surfaces in the quotient $Q:=X/\langle c\rangle$ containing $B\# P_0$, where $B$ is the image of $X^c$ in $Q$ and $P_0\subset S^4$ is the standard $\mathbb{RP}^2$ with normal Euler number $2$.
\end{theorem}

Motivated by Theorem \ref{thm:main}, let us call an admissible pair $(X,c)$ \emph{strongly admissible} if $X$ is non-spin and the complement $X-X^c$ of the fixed locus $X^c$ is simply-connected. Any admissible pair $(X,c)$ admits a Real spin$^c$ structure $\mathfrak{s}$ satisfying \[
    b_+(X)^{-c}=\frac{c_1(\mathfrak{s})^2-\sigma(X)}{8} \quad \text{and} \quad \deg_R(X,c,\mathfrak{s})\neq 0
    \]
by \cite[Proposition 11.2]{Baraglia-2026}. Therefore, a strongly admissible pair $(X,c)$ satisfies the hypotheses of Theorem \ref{thm:main}. Our next result shows that if at least one of the admissible summands is strongly admissible, the additional $\CPbar$ summand on $X$ and the additional $P_0$ summand on $B$ in Theorem \ref{thm:Baraglia} can be omitted. 

\begin{theorem}\label{thm:connected_sum}
    Let $(X_1,c_1),\dots,(X_k,c_k)$ be admissible pairs, and suppose that at least one $(X_i,c_i)$ is strongly admissible. Form the equivariant connected sum \[
    (X,c):=(X_1\# \cdots \# X_k, c_1\# \cdots \# c_k)
    \]
    at fixed points. Then $X$ admits an infinite exotic family of involutions containing $c$. Furthermore, if every fixed locus $X_i^{c_i}$ is connected, then there exists an infinite exotic family of surfaces in the quotient $Q:=X/\langle c\rangle$ containing $B$, where $B$ is the image of $X^c$ in $Q$.
\end{theorem}

\subsection*{Outline of the proofs}
We begin by comparing our construction with Baraglia's. Here are the building blocks:
\begin{itemize}[leftmargin=2em]
    \item Let $P\subset S^3$ be the $(-2,3,7)$-pretzel knot, and let $\rho P=\tau^0\rho^1(P)\subset S^4$ be its $0$-twist $1$-roll spin. The double cover $\Sigma$ of $S^4$ branched along $\rho P$ is a homotopy $4$-sphere \cite[Theorem 4.44]{Miyazawa-2023}. Let $\iota$ denote the covering involution on $\Sigma$.
    \item Let $\kappa$ denote complex conjugation on $\CPbar$, and let $P_0$ denote the standard $\mathbb{RP}^2$ in $S^4$ with normal Euler number $2$; it is the image of the composition $\mathbb{RP}^2\hookrightarrow \CPbar \to \CPbar/\langle \kappa\rangle=S^4$.
\end{itemize}
The following statements were already known: \begin{enumerate}[label=(\arabic*), leftmargin=2em]
    \item \label{item_(1)} For each $n\geq 1$, $P_0\# n(\rho P)$ is topologically isotopic to $P_0$ in $S^4$ \cite[Theorem 4.38]{Miyazawa-2023}.
    \item \label{item_(2)} For each $n\geq 1$, $(\CPbar, \kappa)\#n(\Sigma,\iota)$ is topologically equivalent to $(\CPbar,\kappa)$. (This follows from \ref{item_(1)}.)
    \item \label{item_(3)} $\CPbar\#\Sigma$ is diffeomorphic to $\CPbar$ \cite[Corollary 1.2]{Hughes-Kim-Miller-2024}.
\end{enumerate}
Given a $4$-manifold-involution pair $(X,c)$, Baraglia studies the family \[
\left\{(X,c)\#(\CPbar,\kappa)\# n(\Sigma, \iota)\right\}_{n=0}^\infty 
\]
to get an exotic family of involutions on $X\#\CPbar$, and we shall consider the family \[
\left\{(X,c)\# n(\Sigma,\iota)\right\}_{n=0}^\infty
\]
to get an exotic family of involutions on $X$ itself. For the first family, \ref{item_(2)} directly gives topological equivalence, and \ref{item_(3)} directly gives a diffeomorphism $X\#\CPbar \#n\Sigma\cong X\# \CPbar$. But for the second family, we must establish topological equivalences $(X,c)\#n(\Sigma ,\iota)\cong (X,c)$ and diffeomorphisms $X\#n\Sigma \cong X$. 

Similarly, for the quotient $Q:=X/\langle c\rangle$ with the image $B$ of $X^c$ in $Q$, Baraglia considers the family $\{B\#P_0\#n(\rho P)\}_{n=0}^\infty$ of surfaces, and \ref{item_(1)} directly gives topological isotopy. However, we shall consider the family $\{B\# n(\rho P)\}_{n=0}^\infty$, for which topological isotopy is not direct.

\vspace{1mm}
We now outline the proof of Theorem \ref{thm:main}. To show that $(X\#n\Sigma,c\#n\iota)$ is topologically equivalent to $(X,c)$, we show that the quotient pairs $(Q,B\#n(\rho P))$ and $(Q,B)$ are homeomorphic; then topological equivalence is obtained by lifting this homeomorphism. We first use a trick of Hughes--Kim--Miller \cite{Hughes-Kim-Miller-2024} which gives a diffeomorphism $(Q,B\#\rho P)\cong (Q_P(\varphi),B)$, where $Q_P(\varphi)$ is obtained from $Q$ by \emph{knot surgery} along a torus disjoint from $B$. We then follow the argument of Kim--Ruberman \cite{Kim-Ruberman-2008} to show that $(Q_P(\varphi),B)$ is homeomorphic to $(Q,B)$, using the assumption that the complement $X-X^c$ is simply-connected, and the fact that the knot $P$ has zero Arf invariant. 

To show that $X\# n\Sigma$ is diffeomorphic to $X$, we use Kuhrman's observation \cite{Kuhrman-2025} that $\Sigma$ is obtained from $S^4$ by a \emph{Gluck twist}. Then we apply a criterion of Akbulut--Yasui \cite{Akbulut-Yasui-2013} to conclude that $X\#\Sigma$ is diffeomorphic to $X$, using the assumption that $X$ is non-spin and $X-X^c$ is simply-connected. Iteration yields $X\#n\Sigma \cong X$. 

To prove that the pairs $(X\#n\Sigma, c\# n\iota)$ are pairwise smoothly inequivalent, we use Real Seiberg--Witten theory, based on the results of Miyazawa \cite{Miyazawa-2023} and Baraglia \cite{Baraglia-2026}. More precisely, we use the \emph{Real degree}, which is an invariant of a smooth $4$-manifold-involution pair equipped with a Real spin$^c$ structure. We consider the set of all absolute values of nonzero Real degrees of a pair. The minimum of this set turns out to be pairwise distinct for the pairs $(X\#n\Sigma, c\# n\iota)$, which proves smooth inequivalence. The additional statement that this family remains exotic after equivariant blow-ups is an easy consequence of the blow-up formula for the Real degree.

Finally, to prove that $c\#n\iota$ cannot be anti-symplectic for any orientation-compatible symplectic form, we use the observation (Theorem \ref{thm:symplectic_degree}) that if $(X,\omega)$ is a closed, connected, symplectic $4$-manifold with $b_1(X)=0$ and $b_+(X)\equiv 3\pmod{4}$, and if $c:X\to X$ is an anti-symplectic involution whose fixed locus is a nonempty surface, then the Real degree satisfies $|\deg_R(X,c,\mathfrak{s}_\omega)|=2$ where $\mathfrak{s}_\omega$ is the canonical spin$^c$ structure. This is a Real counterpart of Taubes' result \cite{Taubes-1994} that a closed, connected, symplectic $4$-manifold $(X,\omega)$ with $b_+(X)\geq 2$ satisfies $\operatorname{SW}(X,\mathfrak{s}_\omega)=\pm 1$. We prove that $(X\#n\Sigma,c\#n\iota)$ does not have a Real spin$^c$ structure with Real degree $\pm 2$, thereby completing the proof of Theorem \ref{thm:main}.

\vspace{1mm}
To prove Theorem \ref{thm:isotopy}, we take a little care to show that the homeomorphism of the pairs $(Q,B\#n(\rho P))$ and $(Q,B)$ obtained above can be chosen to induce the identity on $H_*(Q;\mathbb{Z})$. Then the results of Perron \cite{Perron-1986} and Quinn \cite{Quinn-1986} imply that the surfaces $B\# n(\rho P)$ and $B$ are topologically isotopic. The fact that the involutions $c\# n\iota$ are pairwise smoothly inequivalent implies that the surfaces $B\# n(\rho P)$ are pairwise smoothly non-isotopic. 

\vspace{1mm}
For Theorem \ref{thm:connected_sum}, we carry out the construction on a strongly admissible summand. A little extra care is needed to extend an isotopy on the chosen strongly admissible summand to the whole connected sum, and we use the relative isotopy result of Orson--Powell \cite{Orson-Powell-2025}.

\begin{remark}
    If the Miyazawa homotopy $4$-sphere $\Sigma$ is diffeomorphic to the standard $S^4$, then the conclusions of Theorems \ref{thm:main} and \ref{thm:isotopy} hold without the non-spin assumption; this assumption is used only to prove $X\#\Sigma$ is diffeomorphic to $X$.  
\end{remark}

\subsection*{Organization of the paper}
Section \ref{sec:topological_equivalence} establishes topological equivalence of $(X,c)\#n(\Sigma,\iota)$ and $(X,c)$. We first construct homeomorphisms of the quotient-branch pairs and then lift to the double covers. We also prove topological isotopy of the branch surfaces. 

Section \ref{sec:RSW} recalls the Real Seiberg--Witten invariants and the Real degree, mostly following Miyazawa and Baraglia. We also collect the connected-sum formula, the blow-up formula, and some computational results needed later.

Section \ref{sec:main_proofs} proves the main results. After establishing a diffeomorphism $X\#n\Sigma \cong X$ for simply-connected non-spin $X$, we combine the results from Sections \ref{sec:topological_equivalence} and \ref{sec:RSW} to prove Theorems \ref{thm:main}, \ref{thm:isotopy}, and \ref{thm:connected_sum}, together with Corollaries \ref{cor:ordinary_SW_formulation} and \ref{cor:symplectic}.

Sections \ref{sec:branched_covers} and \ref{sec:complete_intersection} are essentially independent of Sections \ref{sec:topological_equivalence}, \ref{sec:RSW}, and \ref{sec:main_proofs}. Section \ref{sec:branched_covers} studies cyclic branched covers of $\mathbb{CP}^2$ and constructs involutions whose fixed loci have simply-connected complements. The main geometric input is the computation of the fundamental group of this complement for a carefully chosen branch curve. We then prove Theorem \ref{thm:branched_cover} and determine the oriented homeomorphism type of the quotients. Section \ref{sec:complete_intersection} proves Theorem \ref{thm:complete_intersection} and determines the oriented homeomorphism type of the quotients.

\subsection*{Acknowledgments}
The authors would like to thank Marco Golla, Marco Marengon, Alice Merz, Jin Miyazawa, Mark Powell, Diego Santoro, and Andr{\'as} Stipsicz for their invaluable comments and suggestions. The first author was supported by ERC Advanced Grant KnotSurf4d. The second author was supported by the grant EXCELLENCE-151337.
The use of AI and LLMs during this project was limited to literature research.
\section{Topological equivalence}\label{sec:topological_equivalence}

For $i=1,2$, let $X_i$ be an oriented topological $4$-manifold with an orientation-preserving involution $c_i:X_i\to X_i$. The pairs $(X_1,c_1)$ and $(X_2,c_2)$ are said to be \emph{topologically equivalent} if there is an orientation-preserving homeomorphism $f:X_1\to X_2$ such that $f\circ c_1=c_2\circ f$. The goal of this section is to prove the following theorem.

\begin{theorem}\label{thm:topological_equivalence}
    Let $X$ be a closed, oriented, smooth $4$-manifold with an orientation-preserving smooth involution $c:X\to X$ such that the fixed locus $X^c$ is a nonempty surface and $X-X^c$ is simply-connected. For each $n\geq 0$, the equivariant connected sum $(X,c)\#n(\Sigma,\iota)=(X\#n\Sigma, c\#n\iota)$ formed at points of the fixed surfaces is topologically equivalent to $(X,c)$. 
\end{theorem}

Here, $\Sigma$ denotes the Miyazawa homotopy $4$-sphere, which is the total space of a double branched cover of $S^4$, and $\iota$ is the covering involution on $\Sigma$; see Section \ref{subsec:Miyazawa} below for details. 

\subsection{Fintushel--Stern knot surgery and twist-roll spun knots}  
For a classical knot $K\subset S^3$, we use the following notation. \begin{itemize}[leftmargin=2em]
    \item $\nu K\subset S^3$ denotes a closed tubular neighborhood of $K$ in $S^3$, and $E(K)$ denotes $S^3-\operatorname{int}\nu K$. 
    \item $\mu_K$ and $\lambda_K$ denote a meridian and a $0$-framed longitude of $K$, respectively, in $\partial E(K)$. 
    \item $\operatorname{Arf}(K)\in \mathbb{Z}_2$ denotes the Arf invariant of $K$. It is determined by the Alexander polynomial $\Delta_K(t)$ of $K$ \cite{Levine-1966}. In fact, \begin{equation}\label{eq:Arf_invariant}
         \operatorname{Arf}(K)=\begin{cases}
        0 & \text{if $\Delta_K(-1)\equiv \pm 1\pmod{8}$}, \\
        1 & \text{if $\Delta_K(-1)\equiv \pm 3\pmod{8}$}.
    \end{cases}
    \end{equation}
    \item For integers $m,n\in \mathbb{Z}$, $\tau^m\rho^n(K)$ denotes the $m$-twist $n$-roll spin of $K$ \cite{Litherland-1979}. In the special case $(m,n)=(0,1)$, we denote $\tau^0\rho^1(K)$ simply by $\rho K$.
    \end{itemize}

We now recall the knot surgery operation on $4$-manifolds, introduced by Fintushel and Stern \cite{Fintushel-Stern-1998}. Let $Z$ be a compact, oriented, smooth $4$-manifold, possibly with boundary. Suppose $T\subset \operatorname{int}Z$ is a smoothly embedded torus with trivial normal bundle. Let $\nu T\cong T\times D^2$ denote a closed tubular neighborhood of $T$ in $\operatorname{int}Z$, and let $\mu_T$ be a meridian of $T$. For a classical knot $K\subset S^3$, let $\varphi:S^1\times \partial E(K)\to \partial \nu T$ be a diffeomorphism satisfying $\varphi_*[\{*\}\times \lambda_K]=[\mu _T]$ on homology. Then the $4$-manifold \[
Z_K(\varphi):= (Z-\operatorname{int}\nu T)\cup_\varphi (S^1\times E(K))
\]
is said to be obtained by \emph{knot surgery} from $Z$. Observe that this operation does not change the boundary of $Z$, so there is a natural identification $\partial Z\cong \partial Z_K(\varphi)$. The following theorem shows that this identification extends to a homeomorphism $Z\cong Z_K(\varphi)$, provided that $\operatorname{Arf}(K)=0$ and both $\pi_1(Z)$ and $\pi_1(Z_K(\varphi))$ are finite cyclic of the same order.

\begin{theorem}\label{thm:Kim-Ruberman}
    Let $Z$ be a compact, connected, oriented, smooth $4$-manifold, possibly with boundary. Let $T\subset \operatorname{int}Z$ be a smoothly embedded torus with a trivial normal bundle. Suppose $K\subset S^3$ is a classical knot with $\operatorname{Arf}(K)=0$. If $\pi_1(Z)\cong \pi_1(Z_K(\varphi))\cong \mathbb{Z}_d$ for some $d\geq 1$, then the natural identification $\partial Z_K(\varphi)\cong  \partial Z$ extends to an orientation-preserving homeomorphism $Z_K(\varphi)\cong Z$. 
\end{theorem}
\begin{proof}
This is essentially the Arf-zero case of the argument given in \cite[Section 2]{Kim-Ruberman-2008}. Although Kim and Ruberman focus on the case where $Z$ is of the form $X-\operatorname{int}\nu \Sigma$ for an oriented $4$-manifold $X$ and a smoothly embedded oriented surface $\Sigma \subset X$, the same argument applies to an arbitrary $Z$ as stated above. So here we only give a sketch. 

Let $O\subset S^3$ denote the unknot. Identify $\partial E(K)$ with $\partial E(O)$ by a diffeomorphism sending $\mu_K$ to $\mu_O$ and $\lambda_K$ to $\lambda_O$. Using this identification, we can also consider the knot surgery $4$-manifold $Z_O(\varphi)$. The proof of \cite[Lemma 2.2]{Kim-Ruberman-2008} then shows that there is an orientation-preserving diffeomorphism $Z\cong Z_O(\varphi)$ rel $Z-\operatorname{int}\nu T$. Therefore, it suffices to show that the natural identification $\partial Z_K(\varphi)\cong \partial Z_O(\varphi)$ extends to an orientation-preserving homeomorphism $Z_K(\varphi)\cong Z_O(\varphi)$. 

Next, the proof of \cite[Lemma 2.3]{Kim-Ruberman-2008} provides a homotopy equivalence $f:Z_K(\varphi)\to Z_O(\varphi)$ such that $f|_{Z-\operatorname{int}\nu T}=\operatorname{id}$. In particular, $f$ restricts to the identity on the boundary. 

Write $Z':=Z_O(\varphi)$. The map $f:Z_K(\varphi)\to Z'$ is a normal map in the sense of \cite{Wall-1970}, with the bundle data given as follows. Choose a homotopy inverse $g:Z'\to Z_K(\varphi)$ rel boundary. If $\nu$ is the stable normal bundle of $Z_K(\varphi)$, the homotopy $gf\simeq \operatorname{id}$ rel boundary gives a stable bundle isomorphism $\nu \cong f^*g^*\nu$, and hence a bundle map covering $f$, with the prescribed normal data on the boundary.

The topological normal cobordism class rel boundary is classified by the invariant \[
n(f)\in H^4(Z',\partial Z';\mathbb{Z})\oplus H^2(Z',\partial Z';\mathbb{Z}_2), 
\]
where the first component of $n(f)$ is $(\sigma(Z_K(\varphi))-\sigma(Z'))/8=0$ \cite{Kirby-Taylor-2001}. The second component of $n(f)$ is $\operatorname{Arf}(K)\operatorname{PD}[T']$ \cite[Lemma 2.4]{Kim-Ruberman-2008}, where $T'\subset S^1\times E(O)\subset Z'$ is the product of the $S^1$ factor with a push-off of $\mu_O$ into $E(O)$. Since $\operatorname{Arf}(K)=0$ by hypothesis, it follows that $f$ is normally cobordant to the identity map of $Z'$. Observe that $f$ is a simple homotopy equivalence. Indeed, since $f$ is a homotopy equivalence with $f|_\partial =\operatorname{id}$, its Whitehead torsion satisfies $\tau(f)=-\overline{\tau(f)}$ in the Whitehead group $\operatorname{Wh}(\mathbb{Z}_d)$. But the involution acts trivially on $\operatorname{Wh}(\mathbb{Z}_d)$ \cite{Bass-1974}, and $\operatorname{Wh}(\mathbb{Z}_d)$ is torsion-free \cite{Cohen-1973}; hence $\tau(f)=0$. Since the Wall group $L^s_5(\mathbb{Z}[\mathbb{Z}_d])$ is trivial, there is no surgery obstruction, so a normal cobordism from $f$ to $\operatorname{id}_{Z'}$ can be surgered to give an $s$-cobordism from $Z_K(\varphi)$ to $Z'$ rel boundary. Such an $s$-cobordism is topologically a product \cite{Freedman-Quinn-1990}. Hence the identification $\partial Z_K(\varphi)\cong \partial Z'$ extends to an orientation-preserving homeomorphism $h:Z_K(\varphi)\cong Z'$, completing the proof. For later use, we note that this $h$ is homotopic to $f$ rel boundary.
\end{proof}

Next, let $Q$ be an oriented smooth $4$-manifold, and let $B\subset Q$ be a smoothly embedded closed surface, possibly nonorientable or disconnected. For a classical knot $K\subset S^3$ and integers $m,n\in \mathbb{Z}$, consider the connected sum \[
(Q,B)\# (S^4,\tau^m \rho^n(K)) \cong (Q,B\#\tau^m \rho^n(K)).
\]
The following proposition, which is a straightforward consequence of \cite[Proposition 2.1]{Hughes-Kim-Miller-2024}, shows that this connected sum can be viewed as a knot surgery.
\begin{proposition}\label{prop:Hughes-Kim-Miller}
    There exists a smoothly embedded torus $T\subset Q-B$ with a trivial normal bundle such that there is an orientation-preserving diffeomorphism of pairs \[
    (Q_K(\varphi_{m,n}),B)\cong(Q, B\#\tau^m\rho^n(K))
    \]
    for some gluing diffeomorphism $\varphi_{m,n}:S^1\times \partial E(K)\to \partial \nu T$ satisfying $(\varphi_{m,n})_*[\{*\}\times \lambda_K]=[\mu_T]$. 
\end{proposition}

Note that the surgery leaves a sufficiently small tubular neighborhood $\nu B$ unchanged, since the torus $T$ is disjoint from $B$. 

\begin{proof}
    Let $S\amalg T\subset S^4$ be the sphere-torus link given in \cite[Proposition 2.1]{Hughes-Kim-Miller-2024}; here $S$ is an unknotted $2$-sphere. There it is proved that, for some diffeomorphism $\varphi_{m,n}:S^1\times \partial E(K)\to \partial \nu T$ satisfying $(\varphi_{m,n})_*[\{*\}\times \lambda_K]=[\mu_T]$, the pair $(S^4_K(\varphi_{m,n}),S)$ is diffeomorphic to $(S^4,\tau^m\rho^n(K))$. Thus, the result follows by taking connected sums with $(Q,B)$ away from the surgery region. 
\end{proof}

\subsection{The Miyazawa homotopy sphere}\label{subsec:Miyazawa}
Let $P$ denote the pretzel knot $P(-2,3,7)$, and consider its roll $\rho P=\tau^0\rho^1(P)\subset S^4$. Let $\Sigma$ denote the double cover of $S^4$ branched along $\rho P$, and let $\iota:\Sigma\to \Sigma$ be its covering involution. Miyazawa proved that $\Sigma$ is a homotopy $4$-sphere \cite[Theorem 4.44]{Miyazawa-2023}. We call $\Sigma$ the \emph{Miyazawa homotopy $4$-sphere}. 

\begin{proposition}[{\cite[Theorem 4.40]{Miyazawa-2023}}]\label{prop:Miyazawa_fundamental_group}
    Let $Q$ be a closed, connected, oriented, smooth $4$-manifold, and let $B\subset Q$ be a smoothly embedded closed surface. If $\pi_1(Q-B)\cong \mathbb{Z}_2$ and a meridian of every component of $B$ represents the nontrivial element, then $\pi_1(Q-(B\# \rho P))\cong \mathbb{Z}_2$ and the meridian of every component of $B\#\rho P$ remains nontrivial.
\end{proposition}

Combining this proposition with the preceding results on knot surgery yields the following corollary.

\begin{corollary}\label{cor:homeomorphism_of_pairs}
    Let $Q$ be a closed, connected, oriented, smooth $4$-manifold, and let $B\subset Q$ be a smoothly embedded closed surface. If $\pi_1(Q-B)\cong \mathbb{Z}_2$ and a meridian of every component of $B$ represents the nontrivial element, then there is an orientation-preserving homeomorphism $(Q,B)\cong (Q,B\# \rho P)$ of pairs.
\end{corollary}
\begin{proof}
    Let $Z:=Q-\operatorname{int}\nu B$, where $\nu B$ is a small closed tubular neighborhood of $B$ in $Q$. Proposition \ref{prop:Hughes-Kim-Miller} gives an orientation-preserving diffeomorphism $F:(Q_P(\varphi),B)\to (Q,B\#\rho P)$ of pairs, where the knot surgery is performed along a torus $T\subset \operatorname{int}Z$ so that $\nu B$ remains unchanged in the knot surgery manifold $Q_P(\varphi)$. Thus, $F(\nu B)$ is a tubular neighborhood $\nu(B\#\rho P)$ of $B\# \rho P$ in $Q$, and $F$ restricts to a diffeomorphism $Z_P(\varphi)\to Q-\operatorname{int}\nu (B\#\rho P)$. It follows from Proposition \ref{prop:Miyazawa_fundamental_group} that both $Z\simeq Q-B$ and \[
    Z_P(\varphi)\cong Q-\operatorname{int} \nu (B\#\rho P) \simeq Q-(B\# \rho P)
    \]
    have fundamental groups isomorphic to $\mathbb{Z}_2$. It is well-known that $\Delta_P(-1)=1$, so the formula \eqref{eq:Arf_invariant} gives $\operatorname{Arf}(P)=0$. Theorem \ref{thm:Kim-Ruberman} therefore shows that the natural identification $\partial Z_P(\varphi)\cong \partial Z$ extends to an orientation-preserving homeomorphism $Z_P(\varphi)\cong Z$. Gluing with $\operatorname{id}_{\nu B}$ gives an orientation-preserving homeomorphism $(Q_P(\varphi),B)\cong (Q,B)$. Since $(Q_P(\varphi),B)\cong (Q,B\#\rho P)$ via $F$, this completes the proof. 
\end{proof}

Now we can give a proof of Theorem \ref{thm:topological_equivalence}.
\begin{proof}[Proof of Theorem \ref{thm:topological_equivalence}]
    Let $Q:=X/\langle c\rangle$. The quotient map $p:X\to Q$ is a double cover branched along $B:=p(X^c)$ \cite[Lemma 7.5]{Hambleton-Hausmann-2011}. Since $\pi_1(X-X^c)=1$ by hypothesis, we have $\pi_1(Q-B)\cong \mathbb{Z}_2$, and a meridian of any component of $B$ represents the nontrivial element of $\pi_1(Q-B)$. 

    For each $n\geq 1$, write $(X_n,c_n):=(X,c)\#n(\Sigma,\iota)$, and let $F_n\subset X_n$ be the fixed locus $(X_n)^{c_n}$. We have $X_n/\langle c_n\rangle =X/\langle c\rangle \# n(\Sigma /\langle \iota\rangle) =Q\# nS^4\cong Q$, and the projection $p_n:X_n\to Q$ is a double cover branched along the connected sum $B\# n(\rho P)$, with covering involution $c_n$. Using Proposition \ref{prop:Miyazawa_fundamental_group} and Corollary \ref{cor:homeomorphism_of_pairs} inductively, we get an orientation-preserving homeomorphism \[
    G_n:(Q,B\# n(\rho P))\longrightarrow (Q,B)
    \]
    for each $n\geq 1$. Elementary covering space theory shows that the restriction $G_n:Q-(B\# n(\rho P))\cong Q-B$ lifts to a homeomorphism $\tilde{G}_n:X_n-F_n\to X-X^c$ satisfying $p\circ \tilde{G}_n=G_n\circ p_n|_{X_n-F_n}$ and $\tilde{G}_n\circ c_n|_{X_n-F_n}=c\circ \tilde{G}_n$. We now extend $\tilde{G}_n$ to a map $X_n\to X$ by defining $\tilde{G}_n(x)$ to be the unique point in $p^{-1}(G_n(p_n(x)))$ for each $x\in F_n$. Then it is straightforward to see that $\tilde{G}_n:X_n\to X$ is an orientation-preserving homeomorphism satisfying $\tilde{G}_n\circ c_n=c\circ \tilde{G}_n$. This completes the proof.     
\end{proof}

The following observation will be used in the proof of Theorem \ref{thm:connected_sum}. 

\begin{remark}\label{rmk:small_ball}
Let $(X,c)$ be a pair satisfying the conditions of Theorem \ref{thm:topological_equivalence}, and let $D\subset X$ be a sufficiently small invariant closed ball centered at a fixed point. If all equivariant connected sums with $(\Sigma,\iota)$ are performed away from $D$, then the topological equivalence $(X\#n\Sigma, c\# n\iota)\cong (X,c)$ of Theorem \ref{thm:topological_equivalence} can be chosen to be the identity on $D$. 

To see this, consider the quotient pair $(Q,B)$ and let $E\subset Q$ denote the image of $D$. We may assume that $E$ is contained in $\operatorname{int}\nu B$. Taking the connected sums away from $E$, the proof of Proposition \ref{prop:Hughes-Kim-Miller} shows that the resulting diffeomorphism $(Q_P(\varphi),B)\cong (Q,B\# \rho P)$ is the identity on $E$. Then the proof of Corollary \ref{cor:homeomorphism_of_pairs} shows that a homeomorphism $(Q,B)\cong (Q,B\#\rho P)$ can be chosen to be the identity on $E$. Repeating this process gives a homeomorphism $(Q,B\# n(\rho P))\cong (Q,B)$ which is the identity on $E$. Now the result follows by choosing the lift to be the identity on $D$ in the proof of Theorem \ref{thm:topological_equivalence}. 

With a little extra care, we can do the same thing for finitely many balls. Suppose $D_1,\dots,D_k\subset X$ are pairwise disjoint small invariant closed balls centered at fixed points, and that all equivariant connected sums with $(\Sigma,\iota)$ are performed away from these balls. We claim that, after shrinking some $D_i$'s if necessary, the topological equivalence $(X\#n\Sigma, c\# n\iota)\cong (X,c)$ of Theorem \ref{thm:topological_equivalence} can be chosen to be the identity on each $D_i$. Let $E_1,\dots,E_k$ be the images of these balls in $Q$. As above, we obtain a homeomorphism $(Q,B\#n(\rho P))\cong (Q,B)$ which is the identity on each $E_i$. But the problem is that we cannot assume that its lift $\tilde{G}:(X\#n\Sigma,c\#n\iota)\cong (X,c)$ is simultaneously the identity on each $D_i$; $\tilde{G}$ is either the identity or the covering involution $c$ on $D_i$. We can correct this as follows. Suppose $\tilde{G}=c$ on $D_i$. Choose local coordinates $(u,z)\in \mathbb{R}^2\times \mathbb{C}$ so that $c(u,z)=(u,-z)$ on $D_i$. Let $\beta:\mathbb{R}\to \mathbb{R}$ be a smooth function that is identically $1$ near $0$ and supported on $(-\epsilon,\epsilon)$ for some small $\epsilon>0$, and define $F_i:D_i\to D_i$ by $F_i(u,z)=(u,e^{i\pi \beta(|u|^2+|z|^2)}z)$; $F_i$ equals $c$ on a smaller ball $D_i'$ in $D_i$ and is the identity outside a slightly larger ball. $F_i$ extends by identity to a diffeomorphism $F_i:X\to X$ commuting with $c$. Now the composition $F_i\circ \tilde{G}$ is a homeomorphism $(X\#n\Sigma,c\#n\iota)\cong (X,c)$ restricting to the identity on $D_i'$. Repeating this finitely many times proves our claim.
\end{remark}

The conclusion of Corollary \ref{cor:homeomorphism_of_pairs} can be strengthened as follows. 
\begin{corollary}\label{cor:isotopy}
    Let $Q$ be a closed, connected, oriented, smooth $4$-manifold, and let $B\subset Q$ be a smoothly embedded closed surface. If $\pi_1(Q-B)\cong \mathbb{Z}_2$ and a meridian of every component of $B$ represents the nontrivial element, then the surfaces $B,B\#\rho P\subset Q$ are topologically isotopic. Furthermore, for a small standard closed ball-disk pair $(E,E\cap B)\subset (Q,B)$, if the connected sum $B\#\rho P$ is performed away from $E$, then $B,B\#\rho P\subset Q$ are topologically isotopic by an isotopy rel $E$. 
\end{corollary}
\begin{proof} Let $Z:=Q-\operatorname{int}\nu B$. The homeomorphism $(Q,B)\cong (Q,B\#\rho P)$ of Corollary \ref{cor:homeomorphism_of_pairs} is given by the composition \[
(Q,B)\xleftarrow[\cong]{F} (Q_P(\varphi),B)\xrightarrow[\cong]{G} (Q,B\#\rho P),
\]
where:  \begin{itemize}[leftmargin=2em]
\item $G$ comes from Proposition \ref{prop:Hughes-Kim-Miller}; it is obtained from a diffeomorphism $(S^4_P(\varphi),S)\cong (S^4,\rho P)$ by taking connected sums with $(Q,B)$. Thus, if $D\subset Q$ is a closed $4$-ball used in forming the connected sum, then $G$ restricts to the identity on $Q-\operatorname{int}D$. 
\item $F$ is obtained by gluing a homeomorphism $h:Z_P(\varphi)\to Z$ such that $h|_\partial=\operatorname{id}$ with the identity map of $\nu B$. The homeomorphism $h$ comes from Theorem \ref{thm:Kim-Ruberman}, and its proof shows that $h$ is homotopic rel boundary to a map $f:Z_P(\varphi)\to Z$ such that $f|_{Z-\operatorname{int}\nu T}=\operatorname{id}$, where $T\subset \operatorname{int}Z$ is the torus where the knot surgery is performed. 
\end{itemize}
Set $F':=f\cup \operatorname{id}_{\nu B}:Q_P(\varphi)\to Q$; it is homotopic to $F$. The composition $F'\circ G^{-1}:Q\to Q$ restricts to the identity on $Q-\operatorname{int}D$, so it induces the identity on homology. Thus, $F\circ G^{-1}$ also induces the identity on homology. Now, note that $Q$ is simply-connected. It was shown by Perron \cite{Perron-1986} and Quinn \cite[Theorem 1.1]{Quinn-1986} that a self-homeomorphism on a simply-connected, closed, topological $4$-manifold inducing the identity on integral homology is topologically isotopic to the identity map. Hence $F\circ G^{-1}$ is topologically isotopic to the identity. This completes the proof of the first statement.

For the second statement, we may assume that $E\subset \operatorname{int}\nu B$ and that the ball $D$ chosen above is disjoint from $E$. Then the homeomorphism $F\circ G^{-1}$ is the identity on $E$, and the restriction of $F\circ G^{-1}$ to $Q^0:=Q-\operatorname{int}E$ induces the identity on homology. Hence $F\circ G^{-1}|_{Q^0}$ is isotopic to $\operatorname{id}_{Q^0}$ rel boundary by \cite[Corollary C(1)]{Orson-Powell-2025}. Extending the isotopy by the identity map of $E$ completes the proof. 
\end{proof}

Combining this with Proposition \ref{prop:Miyazawa_fundamental_group} and using induction gives the following result, which will be used in the proofs of Theorems \ref{thm:isotopy} and \ref{thm:connected_sum}.
\begin{corollary}\label{cor:isotopy_general}
    Let $Q$ be a closed, connected, oriented, smooth $4$-manifold, and let $B\subset Q$ be a smoothly embedded closed surface. If $\pi_1(Q-B)\cong \mathbb{Z}_2$ and a meridian of every component of $B$ represents the nontrivial element, then the surfaces $B\#n(\rho P)\subset Q$ for $n\geq 0$ are all topologically isotopic. Furthermore, for a small standard closed ball-disk pair $(E,E\cap B)\subset (Q,B)$, if the connected sums with $\rho P$ are performed away from $E$, then $B$ and $B\#n(\rho P)\subset Q$ are topologically isotopic by an isotopy rel $E$. 
\end{corollary}

\section{Real Seiberg--Witten theory}\label{sec:RSW}

We recall the basic facts from Real Seiberg--Witten theory on $4$-manifolds; see \cite{Miyazawa-2023, Konno-Miyazawa-Taniguchi-2024, Baraglia-2026} for details. Throughout this section, we assume that $X$ is a closed, connected, oriented, smooth $4$-manifold and $c:X\to X$ is an orientation-preserving smooth involution whose fixed locus is a nonempty surface. Fix a $c$-invariant Riemannian metric $g$ on $X$. We use the following notation. \begin{itemize}[leftmargin=2em]
    \item For $k\geq 0$, $b_k(X)^{-c}$ denotes the dimension of the $(-1)$-eigenspace of $c^*$ on $H^k(X;\mathbb{R})$.
    \item $b_+(X)^{-c}$ denotes the dimension of the $(-1)$-eigenspace $H^2_+(X;\mathbb{R})^{-c^*}$ of $c^*$ on the space $H^2_+(X;\mathbb{R})$ of $g$-self-dual harmonic $2$-forms on $X$. 
    \item For $k\geq 0$, $\Omega^k(X,i\mathbb{R})^{-c^*}$ denotes the space of imaginary-valued $k$-forms $a$ with $c^*a=-a$. 
    \item $\Omega^+(X,i\mathbb{R})^{-c^*}$ denotes the space of imaginary-valued $g$-self-dual $2$-forms $\eta$ on $X$ with $c^*\eta=-\eta$.
\end{itemize}
We also assume throughout the section that $b_1(X)^{-c}=0$.

\subsection{The Real Seiberg--Witten invariant}
Let $\mathfrak{s}$ be a spin$^c$ structure on $X$ with spinor bundle $\mathbb{S}=\mathbb{S}^+\oplus \mathbb{S}^-$ and Clifford multiplication $\rho:T^*X\to \operatorname{End}_{\mathbb{C}}(\mathbb{S})$. We say that $\mathfrak{s}$ is a \emph{Real spin$^c$ structure} on $(X,c)$ if $c$ lifts to an anti-linear fiberwise isometry $I:\mathbb{S}\to \mathbb{S}$ preserving $\mathbb{S}^\pm$ with $I^2=\operatorname{id}$, compatible with the Clifford multiplication in the sense that $I(\rho(a)\psi) = \rho (c^*a) I(\psi)$ for $a\in \Omega^1(X)$ and $\psi\in \Gamma(\mathbb{S})$. There exists such a lift $I$ if and only if $c^*\mathfrak{s}$ is isomorphic to the conjugate $\overline{\mathfrak{s}}$, and there are exactly $2^{b_1(X)}$ distinct lifts up to isomorphism (here we are using that the fixed-point set of $c$ has codimension $2$ and that $b_1(X)^{-c}=0$). We denote by $\operatorname{Spin}^c_R(X,c)$ the set of isomorphism classes of all Real spin$^c$ structures on $(X,c)$. Therefore, the natural map \[
\operatorname{Spin}^c_R(X,c)\longrightarrow \operatorname{Spin}^c(X)
\]
forgetting the lift $I$ is a $2^{b_1(X)}$-to-$1$ map with image $\{\mathfrak{s}\in \operatorname{Spin}^c(X):c^*\mathfrak{s}\cong \overline{\mathfrak{s}}\}$. (The lift $I$ is always part of the data of a Real spin$^c$ structure.)

Suppose now $\mathfrak{s}$ is a Real spin$^c$ structure with a lift $I$ of $c$ as above. A spin$^c$ connection $A$ is said to be \emph{Real} if $I^*A:=I\circ A\circ I=A$. For a fixed Real spin$^c$ connection $A_0$ and $a\in \Omega^1(X,i\mathbb{R})$, we have $I^*(A_0+a)=A_0-c^*a$, and the space $\mathcal{A}_R$ of all Real spin$^c$ connections is the affine space $A_0+\Omega^1(X,i\mathbb{R})^{-c^*}$. For a Real spin$^c$ connection $A$, the Dirac operator $D_A:\Gamma(\mathbb{S}^+)\to \Gamma(\mathbb{S}^-)$ satisfies $I^- D_A=D_AI^+$ where $I^\pm=I|_{\mathbb{S}^\pm}$, and the curvature $F_A\in \Omega^2(X,i\mathbb{R})$ of the induced $\operatorname{U}(1)$-connection on the determinant line bundle $\det \mathbb{S}^+$ satisfies $c^*F_A=-F_A$, i.e., $F_A\in \Omega^2(X,i\mathbb{R})^{-c^*}$. A spinor $\psi\in \Gamma(\mathbb{S})$ is said to be \emph{Real} if $I\psi =\psi$. The space of Real $\pm$-spinors is denoted by $\Gamma(\mathbb{S}^\pm)^I$. For $\eta\in \Omega^2_+(X,i\mathbb{R})^{-c^*}$, the (perturbed) \emph{Real Seiberg--Witten equations} are \begin{equation}\label{eq:perturbed_RSW_equations}
  D_A\psi=0, \quad F_A^+=q(\psi)+\eta \quad ((A,\psi)\in \mathcal{A}_R\times \Gamma(\mathbb{S}^+)^I).  
\end{equation}
Here $q$ is the quadratic map given by $\rho (q(\psi))=(\psi\otimes \psi^*)_0$, where $(\psi\otimes \psi^*)_0$ denotes the trace-free part of $\psi\otimes \psi^*\in \operatorname{End}_{\mathbb{C}}(\mathbb{S}^+)$; the natural extension $\rho:\Lambda^*T^*X\otimes \mathbb{C}\to \operatorname{End}_{\mathbb{C}}(\mathbb{S})$ of the Clifford multiplication identifies $\Lambda^2_+T^*X\otimes i\mathbb{R}$ with the space of traceless Hermitian endomorphisms of $\mathbb{S}^+$. 

The \emph{Real gauge group} $\mathcal{G}_R:=\{u\in C^\infty(X,S^1):u\circ c=\bar{u}\}$ acts on $\mathcal{A}_R\times \Gamma(\mathbb{S}^+)^I$ by $u\cdot (A,\psi)=(A-u^{-1}du, u\psi)$. The Real Seiberg--Witten equations are $\mathcal{G}_R$-equivariant. A Real Seiberg--Witten solution $(A,\psi)$ is said to be \emph{reducible} if $\psi\equiv 0$ and \emph{irreducible} otherwise. The stabilizer of a reducible solution is $\{\pm 1\}$, and irreducible solutions have trivial stabilizer. The moduli space of solutions modulo $\mathcal{G}_R$ is denoted by $\mathcal{M}_R:=\mathcal{M}_R(X,c,\mathfrak{s},g,\eta)$. Its (real) expected dimension is $d(\mathfrak{s})-b_+(X)^{-c}$, where \[
d(\mathfrak{s}):=\operatorname{ind}_{\mathbb{C}}D_A=\frac{c_1(\mathfrak{s})^2-\sigma(X)}{8}.
\]
If $b_+(X)^{-c}>0$, there are no reducible solutions for a generic choice of $\eta$, and $\mathcal{M}_R$ is a smooth compact manifold of dimension $d(\mathfrak{s})-b_+(X)^{-c}$. When $b_+(X)^{-c}>1$, the cobordism class of $\mathcal{M}_R$ is independent of the choice of $g$ and $\eta$; when $b_+(X)^{-c}=1$ a chamber has to be chosen. The quotient space $\mathcal{B}_R^*:=(\mathcal{A}_R\times (\Gamma(\mathbb{S}^+)^I-\{0\}))/\mathcal{G}_R$ has the homotopy type of $\mathbb{RP}^\infty$. If $u$ is the generator of $H^1(\mathcal{B}_R^*;\mathbb{Z}_2)$, then the \emph{mod-$2$ Real Seiberg--Witten invariant} is defined by \[
\operatorname{SW}_R(X,c,\mathfrak{s}):=\left\langle u^{d(\mathfrak{s})-b_+(X)^{-c}}, [\mathcal{M}_R]\right\rangle \in \mathbb{Z}_2,
\]
assuming $d(\mathfrak{s})\geq b_+(X)^{-c}>1$. In particular, if $d(\mathfrak{s})=b_+(X)^{-c}$, then $\operatorname{SW}_R(X,c,\mathfrak{s})$ is exactly the number of gauge equivalence classes of solutions modulo $2$. If $d(\mathfrak{s})=b_+(X)^{-c}>0$ is even, then the moduli space $\mathcal{M}_R$ can be oriented, so we can define the \emph{integer-valued Real Seiberg--Witten invariant} \begin{equation}\label{eq:integer_SW}
    \operatorname{SW}_{R,\mathbb{Z}}(X,c,\mathfrak{s})=\sum_{[A,\psi]\in \mathcal{M}_R} \operatorname{sign}\,[A,\psi]\in \mathbb{Z},
\end{equation}
which is well-defined up to an overall sign, and reduces to $\operatorname{SW}_R(X,c,\mathfrak{s})$ modulo $2$. 

\subsection{The Real degree}
We now consider the Real version of the Bauer--Furuta invariant \cite{Bauer-Furuta-2004}. The identity component $\mathcal{G}_R^0$ of $\mathcal{G}_R$ is given by $\{e^{if}:f\in C^\infty(X,\mathbb{R}), \,f\circ c=-f\}$, and $\mathcal{G}_R/\mathcal{G}_R^0\cong \mathbb{Z}_2=\{\pm 1\}$. For the restricted $\mathcal{G}_R^0$-action on $\mathcal{A}_R$, every orbit has a unique representative $A\in \mathcal{A}_R$ satisfying $d^*(A-A_0)=0$. We may therefore work on the global Coulomb slice $\ker d^* \oplus \Gamma(\mathbb{S}^+)^I\subset \Omega^1(X,i\mathbb{R})^{-c^*}\oplus \Gamma(\mathbb{S})^I$. For a fixed integer $k\geq 4$, take Sobolev completions \begin{align*}
    \mathcal{U}_k&:=\left\{a\in L^2_k(X,T^*X\otimes i\mathbb{R}):c^*a=-a, \, d^*a=0 \right\}, \\
    \mathcal{U}'_{k-1}&:=L^2_{k-1}(X,\Lambda^2_+T^*X\otimes i\mathbb{R})^{-c^*}= \left\{\eta \in L^2_{k-1}(X,\Lambda^2_+T^*X\otimes i\mathbb{R}):c^*\eta=-\eta\right\}, \\
    \mathcal{V}_k^+&:=L^2_k(X,\mathbb{S}^+)^I=\left\{\psi\in L^2_k(X,\mathbb{S}^+):I\psi=\psi \right\}, \\  \mathcal{V}_{k-1}^-&:=L^2_{k-1}(X,\mathbb{S}^-)^I=\left\{ \psi\in L^2_{k-1}(X,\mathbb{S}^-):I\psi=\psi\right\}.
\end{align*}
The \emph{Real monopole map} (in Coulomb gauge) is defined by \begin{align*}
    \mu_R: \mathcal{U}_k\oplus \mathcal{V}_k^+ & \longrightarrow \mathcal{U}_{k-1}'\oplus \mathcal{V}_{k-1}^-, \\
    (a,\psi) &\longmapsto \left(2d^+a-q(\psi)  ,\, D_{A_0}\psi+\rho(a)\psi\right).
\end{align*}
Note that $(a,\psi)$ is a solution to \eqref{eq:perturbed_RSW_equations} if and only if $\mu_R(a,\psi)=(\eta-F_{A_0}^+,0)$. The map $\mu_R$ is a compact perturbation of the real Fredholm map \[
2d^+\oplus D_{A_0}^I:\mathcal{U}_k\oplus \mathcal{V}_k^+ \longrightarrow \mathcal{U}_{k-1}'\oplus \mathcal{V}_{k-1}^-,
\]
where $D_{A_0}^I$ denotes the restriction of $D_{A_0}$ to Real spinors. We have \[
\ker d^+=0, \quad \operatorname{coker}d^+\cong H^2_+(X;i\mathbb{R})^{-c^*}, \quad \operatorname{ind}_{\mathbb{R}}D_{A_0}^I=\operatorname{ind}_{\mathbb{C}}D_{A_0}=d(\mathfrak{s}),
\]
and in particular, $\operatorname{ind}_{\mathbb{R}}(2d^+\oplus D_{A_0}^I)=d(\mathfrak{s})-b_+(X)^{-c}$. The map $\mu_R$ is $\mathbb{Z}_2$-equivariant, where the nontrivial element of $\mathbb{Z}_2$ acts trivially on $\mathcal{U}_k$, $\mathcal{U}_{k-1}'$, and acts as $-1$ on $\mathcal{V}_k^+$, $\mathcal{V}_{k-1}^-$.

For the finite-dimensional approximation, choose sufficiently large finite-dimensional subspaces $U'\subset \mathcal{U}_{k-1}'$ and $V'\subset \mathcal{V}_{k-1}^-$, so that the natural maps $U'\to \operatorname{coker}d^+$ and $V'\to \operatorname{coker}D_{A_0}^I$ are surjective. Set $U:=(d^+)^{-1}(U')$ and $V:=(D_{A_0}^I)^{-1}(V')$; then $U$, $V$ are finite-dimensional real vector spaces with $\dim_{\mathbb{R}}U=\dim_{\mathbb{R}}U'-b_+(X)^{-c}$ and $\dim_{\mathbb{R}}V=\dim_{\mathbb{R}}V'+d(\mathfrak{s})$. Then the procedure of \cite{Bauer-Furuta-2004} gives a based $\mathbb{Z}_2$-equivariant map \[
f:S^{U\oplus V}\longrightarrow S^{U'\oplus V'}
\]
between the one-point compactifications. If we assume that $d(\mathfrak{s})=b_+(X)^{-c}$, then the source and target spheres have the same dimension. After choosing orientations on the vector spaces, the \emph{Real degree} is defined by forgetting the $\mathbb{Z}_2$-actions and taking the mapping degree of $f$: \[
\operatorname{deg}_R(X,c,\mathfrak{s}):=\deg (f)\in \mathbb{Z}.
\]
Its absolute value $|\deg_R(X,c,\mathfrak{s})|$ is independent of all the auxiliary choices that have been made, so it is an invariant of $(X,c,\mathfrak{s})$. It has the following properties \cite[Proposition 5.9]{Baraglia-2026}.
\begin{itemize}[leftmargin=2em]
    \item $\deg_R(X,c,\mathfrak{s})\neq 0$ implies $d(\mathfrak{s})$ is even. Indeed, since $\mathbb{Z}_2$ acts trivially on $U,U'$ and by $-1$ on $V,V'$, the $\mathbb{Z}_2$-equivariance of $f$ implies that $(-1)^{\dim_{\mathbb{R}}V} \deg (f)=(-1)^{\dim_{\mathbb{R}}V'}\deg(f)$; hence $\deg(f)\neq 0$ implies $\dim_{\mathbb{R}}V-\dim_{\mathbb{R}}V'=d(\mathfrak{s})$ is even.
    \item If $d(\mathfrak{s})=b_+(X)^{-c}$ is a positive even integer, then \begin{equation}\label{eq:degree_and_SW}
        |\deg_R(X,c,\mathfrak{s})|=2|\operatorname{SW}_{R,\mathbb{Z}}(X,c,\mathfrak{s})|.
    \end{equation}
    For a generic perturbation, the solution space in the global Coulomb slice is a double cover of $\mathcal{M}_R$, and the two points of each class in $\mathcal{M}_R$ have the same sign since $d(\mathfrak{s})$ is even. This explains the factor $2$.
\end{itemize}

\subsection{Computations}
We now summarize computational results that will be used in the proofs of the main results. We begin with the formula for equivariant connected sums. 
\begin{theorem}[{\cite[Theorem 9.2]{Baraglia-2026}}]\label{thm:connected_sum_formula}
    For $i=1,2$, let $X_i$ be a closed, connected, oriented, smooth $4$-manifold and $c_i:X_i\to X_i$ be an orientation-preserving smooth involution whose fixed locus is a nonempty surface with $b_1(X_i)^{-c_i}=0$. Suppose $\mathfrak{s}_i$ is a Real spin$^c$ structure on $(X_i,c_i)$ with $d(\mathfrak{s}_i)=b_+(X_i)^{-c_i}\in 2\mathbb{Z}$. Then \[
    |\deg_R(X_1\#X_2,c_1\#c_2,\mathfrak{s}_1\#\mathfrak{s}_2)|=|\deg_R(X_1,c_1,\mathfrak{s}_1)|\cdot |\deg_R(X_2,c_2,\mathfrak{s}_2)|.
    \]
\end{theorem}

See also \cite[Corollary 2.14]{Miyazawa-2023}. For complex conjugation on $\CPbar$, we have the following result.

\begin{theorem}[{\cite[Proposition 4.33]{Miyazawa-2023}}]\label{thm:conjugation} Let $\kappa:\CPbar\to \CPbar$ be complex conjugation. There exists a Real spin$^c$ structure $\mathfrak{t}$ on $(\CPbar,\kappa)$ with $c_1(\mathfrak{t})^2=-1$ such that $d(\mathfrak{t})=b_+(\CPbar)^{-\kappa}=0$ and \[
|\deg_R(\CPbar,\kappa,\mathfrak{t})|=1. 
\] 
\end{theorem}

Combining Theorems \ref{thm:connected_sum_formula} and \ref{thm:conjugation} yields the following formula for equivariant blow-ups.
\begin{theorem}\label{thm:blow_up_formula}
    Let $X$ be a closed, connected, oriented, smooth $4$-manifold and $c:X\to X$ be an orientation-preserving smooth involution whose fixed locus is a nonempty surface. Suppose $\mathfrak{s}$ is a Real spin$^c$ structure on $(X,c)$ with $d(\mathfrak{s})=b_+(X)^{-c}\in 2\mathbb{Z}$. Then \[
    |\deg_R(X\#\CPbar,c\#\kappa,\mathfrak{s}\#\mathfrak{t})|=|\deg_R(X,c,\mathfrak{s})|. \]
\end{theorem}

For the Miyazawa homotopy $4$-sphere $(\Sigma,\iota)$ (see Section \ref{subsec:Miyazawa}), we have the following result. 
\begin{theorem}[{\cite[Theorem 4.32]{Miyazawa-2023}}]\label{thm:Miyazawa_sphere}
    The Miyazawa sphere $(\Sigma,\iota)$ has a unique Real spin$^c$ structure $\mathfrak{s}_\Sigma$, and it satisfies $d(\mathfrak{s}_\Sigma)=b_+(\Sigma)^{-\iota}=0$ and \[
    |\deg_R(\Sigma,\iota,\mathfrak{s}_\Sigma)|=3.
    \]
\end{theorem}

Combining this with Theorem \ref{thm:connected_sum_formula} gives the following corollary; see also the proof of \cite[Theorem 11.4]{Baraglia-2026}.

\begin{corollary}\label{cor:equivariant_connected_sum_degree}
    Let $X$ be a closed, connected, oriented, smooth $4$-manifold and $c:X\to X$ be an orientation-preserving smooth involution whose fixed locus is a nonempty surface with $b_+(X)^{-c}\in 2\mathbb{Z}$. Then for each integer $n\geq 1$, the connected sum map \begin{align*}
        \operatorname{Spin}^c_R(X,c)&\longrightarrow \operatorname{Spin}^c_R(X\#n\Sigma, c\#n\iota), \\
        \mathfrak{s} & \longmapsto \mathfrak{s}\# n\mathfrak{s}_\Sigma,
    \end{align*}
    is a bijection. Moreover, if $\mathfrak{s}$ is a Real spin$^c$ structure on $(X,c)$ satisfying $d(\mathfrak{s})=b_+(X)^{-c}$, then $\mathfrak{s}\#n\mathfrak{s}_\Sigma$ also satisfies $d(\mathfrak{s}\# n\mathfrak{s}_\Sigma)=b_+(X\#n\Sigma)^{-(c\#n\iota)}$, and \[
     |\deg_R(X\#n\Sigma ,c\#n\iota,\mathfrak{s}\#n\mathfrak{s}_\Sigma)|=3^n|\deg_R(X,c,\mathfrak{s})|.
    \]
\end{corollary}

Finally, for anti-symplectic involutions, we have the following result.
\begin{theorem}\label{thm:symplectic_degree}
    Let $X$ be a closed, connected, oriented, smooth $4$-manifold and $c:X\to X$ be an orientation-preserving smooth involution whose fixed locus is a nonempty surface. Suppose $b_1(X)=0$ and $\omega$ is an orientation-compatible symplectic form on $X$ with $c^*\omega=-\omega$. Then the canonical spin$^c$ structure $\mathfrak{s}_\omega$ is a Real spin$^c$ structure on $(X,c)$ with \[
    d(\mathfrak{s}_\omega)=b_+(X)^{-c}=\frac{b_+(X)+1}{2}>0.
    \]
    If this integer is even, then \[ 
    |\operatorname{SW}_{R,\mathbb{Z}}(X,c,\mathfrak{s}_\omega)|=1 \quad \text{and} \quad     |\deg_R(X,c,\mathfrak{s}_\omega)|=2.
    \]
\end{theorem}
\begin{proof}
    For the first statement, see the proof of \cite[Proposition 11.2]{Baraglia-2026}. For the second statement, assume that $d(\mathfrak{s}_\omega)$ is even; in particular, $b_+(X)>1$. Choose an $\omega$-compatible almost complex structure $J$ so that $c$ is $J$-anti-holomorphic, i.e., $dc\circ J=-J\circ dc$. We use the associated Riemannian metric $\omega(-,J-)$. Now we examine the result of Taubes \cite{Taubes-1994} that $\operatorname{SW}(X,\mathfrak{s}_\omega)=\pm 1$.
    
    For the canonical spin$^c$ structure $\mathfrak{s}_\omega$, the $\pm$-spinor bundles are $\mathbb{S}^+=\underline{\mathbb{C}}\oplus \Lambda^{0,2}_J$ and $\mathbb{S}^-=\Lambda^{0,1}_J$, where $\underline{\mathbb{C}}$ is the trivial bundle $X\times \mathbb{C}$. The formula $I(\alpha):=\overline{c^*\alpha}$ defines a Real lift for $\mathfrak{s}_\omega$, and the spinor $u_0:=(1,0)\in \Gamma(\mathbb{S}^+)$ is Real. There exists a unique spin$^c$ connection $A_0$ such that $\operatorname{pr}_{\underline{\mathbb{C}}}(\nabla_{A_0}u_0)=0$, where $\operatorname{pr}_{\underline{\mathbb{C}}}$ is the projection $\mathbb{S}^+\to \underline{\mathbb{C}}$, and its Dirac operator $D_{A_0}$ satisfies $D_{A_0}u_0=0$; see \cite[Section 1]{Taubes-1994} or \cite[Lemma 4.4]{Hutchings-Taubes-1999}. The spin$^c$ connection $I^*A_0$ also satisfies $\operatorname{pr}_{\underline{\mathbb{C}}}(\nabla_{I^*A_0}u_0)=0$, so uniqueness implies that $I^*A_0=A_0$, i.e., $A_0$ is a Real spin$^c$ connection. For $t>0$, set $\eta_t:=F_{A_0}^+-tq(u_0)\in \Omega^+(X,i\mathbb{R})^{-c^*}$. The pair $(A_0,\sqrt{t}u_0)$ is a Real irreducible solution to the perturbed (ordinary) Seiberg--Witten equations: \begin{equation}\label{eq:perturbed_SW}
        D_A\psi =0, \quad F_A^+=q(\psi)+\eta_t.
    \end{equation}
    Moreover, for sufficiently large $t$, the following statements hold: \begin{itemize}[leftmargin=2em]
        \item $(A_0,\sqrt{t}u_0)$ is the only solution to the equations \eqref{eq:perturbed_SW} up to gauge transformation.
        \item $(A_0,\sqrt{t}u_0)$ is a regular solution. This is equivalent to saying that the linearization of the Seiberg--Witten map on the global Coulomb slice (under suitable Sobolev completions) \begin{align}\label{eq:linearization} \begin{split}
          \mu:   \ker d^* \oplus \Gamma(\mathbb{S}^+) & \longrightarrow \Omega^2_+(X,i\mathbb{R})\oplus \Gamma(\mathbb{S}^-), \\
            (a,\psi) & \longmapsto \left(F_{A_0}^++2d^+a-q(\psi)-\eta_t, \, D_{A_0}\psi+\rho(a)\psi \right),
        \end{split}           
        \end{align}
        at $(0,\sqrt{t}u_0)$ is surjective, where $\rho$ is Clifford multiplication and $\ker d^*\subset \Omega^1(X,i\mathbb{R})$.
    \end{itemize}
    (Observe that these two statements imply $\operatorname{SW}(X,\mathfrak{s}_\omega)=\pm 1$. The perturbed equations \eqref{eq:perturbed_SW} were given in \cite[Section 5]{Taubes-1995}; the original proof in \cite{Taubes-1994} contains a more complicated argument. Detailed proofs of the two statements above can be found in \cite[pp.432-436]{Salamon-1999}.) 

    Using these, we can prove that the following statements hold for sufficiently large $t$: \begin{itemize}[leftmargin=2em]
        \item $(A_0,\sqrt{t}u_0)$ is the only Real solution to the equations \eqref{eq:perturbed_SW} up to Real gauge transformation.
        \item $(A_0,\sqrt{t}u_0)$ is regular as a Real solution. That is, the linearization of the map (again, under suitable Sobolev completions) \begin{align}\label{eq:linearization_Real} \begin{split}
          \mu^I: (\ker d^*)^{-c^*} \oplus \Gamma(\mathbb{S}^+)^I & \longrightarrow \Omega^2_+(X,i\mathbb{R})^{-c^*}\oplus \Gamma(\mathbb{S}^-)^I, \\
            (a,\psi) & \longmapsto \left(F_{A_0}^++2d^+a-q(\psi)-\eta_t, \, D_{A_0}\psi+\rho(a)\psi \right),
        \end{split}    
        \end{align}
        at $(0,\sqrt{t}u_0)$ is surjective. 
    \end{itemize}
    The first statement follows from an elementary observation that if two Real irreducible solutions are gauge equivalent by an ordinary gauge transformation $u$, then $u$ is a Real gauge transformation. For the second statement, observe that the map $\mu$ in \eqref{eq:linearization} is equivariant in the sense that $\mu(a,\psi)=(\xi,\phi)$ implies $\mu(-c^*a,I\psi)=(-c^*\xi, I\phi)$. It follows that its linearization \[
    L:\ker d^* \oplus \Gamma(\mathbb{S}^+)  \longrightarrow \Omega^2_+(X,i\mathbb{R})\oplus \Gamma(\mathbb{S}^-)
    \] 
    at $(0,\sqrt{t}u_0)$ is also equivariant, and that the linearization \[
    L^I:  (\ker d^*)^{-c^*} \oplus \Gamma(\mathbb{S}^+)^I  \longrightarrow \Omega^2_+(X,i\mathbb{R})^{-c^*}\oplus \Gamma(\mathbb{S}^-)^I
    \]
    of the map $\mu^I$ in \eqref{eq:linearization_Real} at $(0,\sqrt{t}u_0)$ is the restriction of $L$. The surjectivity of $L^I$ is now obvious: for any $(\xi,\phi)$ in the target, choose $(a,\psi)\in \ker d^*\oplus \Gamma(\mathbb{S}^+)$ with $L(a,\psi)=(\xi,\phi)$ using the surjectivity of $L$; then $(a-c^*a,\psi+I\psi)\in (\ker d^*)^{-c^*} \oplus \Gamma(\mathbb{S}^+)^I $ and $L^I(a-c^*a,\psi+I\psi)=2(\xi,\phi)$. This proves the second statement.

    The two statements imply that the Real moduli space $\mathcal{M}_R$ defined by \eqref{eq:perturbed_SW} is cut out transversely and consists of a single point. Consequently, the integer-valued Real Seiberg--Witten invariant \eqref{eq:integer_SW} must be $\operatorname{SW}_{R,\mathbb{Z}}(X,c,\mathfrak{s}_\omega)=\pm 1$. Finally, $|\deg_R(X,c,\mathfrak{s}_\omega)|=2$ follows from \eqref{eq:degree_and_SW}. \end{proof}

\section{Proofs of the main results}\label{sec:main_proofs}

In this section, we prove the main results of the paper. We begin with the following proposition.

\begin{proposition}\label{prop:Gluck_twist}
    Let $X$ be a closed, simply-connected, oriented, non-spin, smooth $4$-manifold. Then for every $n\geq 0$, there is an orientation-preserving diffeomorphism $X\# n\Sigma \cong X$, where $\Sigma$ is the Miyazawa sphere described in Section \ref{subsec:Miyazawa}. 
\end{proposition}

To prove this proposition, we recall the Gluck twist operation on $4$-manifolds \cite{Gluck-1962}. Let $M$ be a closed, connected, oriented, smooth $4$-manifold, and let $S\subset M$ be a smoothly embedded $2$-sphere with a trivial normal bundle. Choose an identification $\nu S\cong S^2\times D^2$, and let $\gamma:S^2\times S^1\to S^2\times S^1$ be the diffeomorphism defined by \[
\gamma(x,e^{i\theta})=(R_\theta(x),e^{i\theta}),
\]
where $R_\theta$ is the rotation through angle $\theta$ about a fixed axis. Then the \emph{Gluck twist} of $M$ along $S$ is defined to be the $4$-manifold \[
M_S:=(M-\operatorname{int}\nu S)\cup _\gamma (S^2\times D^2).
\]
The following result of Akbulut and Yasui gives a sufficient condition for a Gluck twist to preserve the oriented diffeomorphism type of $M$.
\begin{theorem}[{\cite[Corollary 1.3]{Akbulut-Yasui-2013}}]\label{thm:Akbulut-Yasui}
    If $M-\operatorname{int}\nu S$ contains a compact, simply-connected, codimension-zero submanifold whose intersection form is odd, then there exists an orientation-preserving diffeomorphism $M_S\cong M$. 
\end{theorem}

\begin{proof}[Proof of Proposition \ref{prop:Gluck_twist}]
    It suffices to prove that $X\#\Sigma \cong X$. Kuhrman showed that there exists a smoothly embedded $2$-sphere $S$ in $S^4$ such that $\Sigma$ is orientation-preservingly diffeomorphic to the Gluck twist $(S^4)_S$ \cite[Corollary 2.9]{Kuhrman-2025}. Taking connected sum with $X$ away from the surgery region then gives $X\# \Sigma \cong X\# (S^4)_S \cong (X\#S^4)_S$. If $B\subset X$ is a closed $4$-ball used in forming the connected sum, then $X-\operatorname{int}B$ is a compact, simply-connected, codimension-zero submanifold of $(X\#S^4)-\operatorname{int}\nu S$ with odd intersection form. Theorem \ref{thm:Akbulut-Yasui} therefore gives an orientation-preserving diffeomorphism $(X\# S^4)_S\cong X\#S^4\cong X$. This completes the proof.
\end{proof}

We are now ready to prove Theorems \ref{thm:main} and \ref{thm:isotopy}. For a closed, connected, oriented, smooth $4$-manifold $X$ with an orientation-preserving smooth involution $c$ whose fixed locus is a nonempty surface and $b_1(X)^{-c}=0$, define \[
\mathcal{D}(X,c):=\left\{|\deg_R(X,c,\mathfrak{s})|\neq 0:\mathfrak{s}\in \operatorname{Spin}^c_R(X,c), \, d(\mathfrak{s})=b_+(X)^{-c} \right\}\subset \mathbb{Z}_{>0},
\]
where $d(\mathfrak{s})=(c_1(\mathfrak{s})^2-\sigma(X))/8$ as in Section \ref{sec:RSW}. Then define \[
m(X,c):=\begin{cases}
    \min \mathcal{D}(X,c) & \text{if $\mathcal{D}(X,c)$ is nonempty}, \\
    0 & \text{otherwise}.
\end{cases}
\]
Clearly, $\mathcal{D}(X,c)$ and $m(X,c)$ are invariants of the smooth equivalence class of $(X,c)$. 

\begin{proof}[Proof of Theorem \ref{thm:main}]
    The hypothesis $\pi_1(X-X^c)=1$ implies that $\pi_1(X)=1$. In particular, $b_1(X)=0$ and $b_1(X)^{-c}=0$. Moreover, the hypothesis $\deg_R(X,c,\mathfrak{s})\neq 0$ implies that $m(X,c)>0$ and $b_+(X)^{-c}$ is even. 

    For each $n\geq 1$, let $(X_n,\hat{c}_n)$ denote the equivariant connected sum $(X,c)\#n(\Sigma,\iota)$. Proposition \ref{prop:Gluck_twist} gives an orientation-preserving diffeomorphism $g_n:X_n\to X$. Set $c_n:=g_n\hat{c}_ng_n^{-1}:X\to X$ and $c_0:=c$. By Theorem \ref{thm:topological_equivalence}, there exists an orientation-preserving homeomorphism $\Phi_n:X_n\to X$ satisfying $\Phi_n\hat{c}_n=c\Phi_n$. Thus $(\Phi_ng_n^{-1})c_n=c(\Phi_ng_n^{-1})$; hence the involution $c_n$ is topologically equivalent to $c$. 

    Next, we have $\mathcal{D}(X,c_n)=\mathcal{D}(X_n,\hat{c}_n)=3^n\cdot \mathcal{D}(X,c)$ by Corollary \ref{cor:equivariant_connected_sum_degree}, and therefore $m(X,c_n)=3^n\cdot m(X,c)$. Since $m(X,c)>0$, the numbers $m(X,c_n)$ ($n=0,1,\dots$) are pairwise distinct. Thus, the involutions $c_n$ ($n=0,1,\dots$) are pairwise smoothly inequivalent. 

    Now, suppose that, for $n\geq 1$, there is an orientation-compatible symplectic form $\omega$ on $X$ with $c_n^*\omega=-\omega$. We clearly have $b_+(X)^{-c_n}=b_+(X_n)^{-\hat{c}_n}=b_+(X)^{-c}\in 2\mathbb{Z}$, so Theorem \ref{thm:symplectic_degree} shows that $2\in \mathcal{D}(X,c_n)$. But this is impossible since $\mathcal{D}(X,c_n)=3^n\cdot \mathcal{D}(X,c)$ so that every element of $\mathcal{D}(X,c_n)$ is a multiple of $3^n$. 

    Finally, fix $k\geq 1$ and let $\kappa$ be complex conjugation on $\CPbar$. Let $D\subset X$ be a small $c$-invariant closed ball centered at a fixed point, away from the connected sum regions with $(\Sigma,\iota)$'s. Remark \ref{rmk:small_ball} shows that the homeomorphisms $\Phi_n$ can be chosen to be the identity on $D$. Performing the equivariant blow-ups in $D$, we obtain homeomorphisms $\Phi_n\# k\operatorname{id}:X_n\#k\CPbar\to X\#k\CPbar$. It follows that all $c_n\# k\kappa$ are topologically equivalent to $c\# k\kappa$. Next, for the Real spin$^c$ structure $\mathfrak{s}$ on $(X,c)$ given in the hypothesis, Theorem \ref{thm:blow_up_formula} shows that \[
    |\deg_R(X\#k\CPbar,c\#k\kappa,\mathfrak{s}\#k\mathfrak{t})|=|\deg_R(X,c,\mathfrak{s})|>0.
    \]
    In particular, the set $\mathcal{D}(X\#k\CPbar,c\#k\kappa)$ is nonempty. Therefore, the same argument applied to $(X,c)$ above shows that the involutions $c_n\#k\kappa$ ($n=0,1,\dots$) are pairwise smoothly inequivalent. This completes the proof.
\end{proof}

\begin{proof}[Proof of Theorem \ref{thm:isotopy}]
    Set $B_0=B$ and set $B_n=B\#n(\rho P)\subset Q$ for each $n\geq 1$. The pair $(Q,B)$ satisfies the hypotheses of Corollary \ref{cor:isotopy_general} (see the proof of Theorem \ref{thm:topological_equivalence}), so each $B_n$ is topologically isotopic to $B_0$. On the other hand, the involution $c_n:X\to X$ constructed in the proof of Theorem \ref{thm:main} is precisely the covering involution of the double cover $X\to Q$ branched along $B_n$. Since $c_n$ and $c_m$ are smoothly inequivalent for $n\neq m$, the pairs $(Q,B_n)$ and $(Q,B_m)$ are not diffeomorphic for $n\neq m$. In particular, $B_n$ and $B_m$ are not smoothly isotopic if $n\neq m$. 
\end{proof}

To prove Corollaries \ref{cor:ordinary_SW_formulation} and \ref{cor:symplectic}, we recall the definitions of admissibility and strong admissibility given in the introduction. Let $X$ be a closed, oriented, smooth $4$-manifold, and let $c:X\to X$ be an orientation-preserving smooth involution. The pair $(X,c)$ is \emph{admissible} if $b_1(X)^{-c}=0$, the fixed locus $X^c$ is a nonempty surface, and $(X,c)$ satisfies one of the following conditions. \begin{enumerate}[label=(\arabic*), leftmargin=2em]
    \item \label{item:1} $X$ has an orientation-compatible symplectic form $\omega$ with $c^*\omega=-\omega$ and $b_+(X)-b_1(X)\equiv 3 \pmod{4}$. 
    \item $X$ has a spin structure $\mathfrak{s}$ with $c^*\mathfrak{s}\cong \mathfrak{s}$ and $b_+(X)^{-c}=\sigma(X)=0$. 
    \item \label{item:3} $X$ has a spin$^c$ structure $\mathfrak{s}$ with $c^*\mathfrak{s}\cong \overline{\mathfrak{s}}$, $\operatorname{SW}(X,\mathfrak{s})$ is odd, $b_+(X)-b_1(X)\equiv 3 \pmod{4}$, and \[
      d(\mathfrak{s})=\frac{b_+(X)-b_1(X)+1}{2}=b_+(X)^{-c}.
    \]
    \item $X=\CPbar$ with $b_+(X)^{-c}=0$. 
    \item $X=N\# N$, and $c$ is the involution interchanging the two factors, where $N$ is negative definite, $b_1(N)=0$, and $N$ has a spin$^c$ structure $\mathfrak{s}$ with $c_1(\mathfrak{s})^2=-b_2(N)$.
\end{enumerate}

\begin{proposition}[{\cite[Proposition 11.2]{Baraglia-2026}}]\label{prop:admissible_pair} Every admissible pair $(X,c)$ admits a Real spin$^c$ structure $\mathfrak{s}$ satisfying $ b_+(X)^{-c}=d(\mathfrak{s})\in 2\mathbb{Z}$ and $\deg_R(X,c,\mathfrak{s})\neq 0$. 
\end{proposition}

An admissible pair $(X,c)$ is \emph{strongly admissible} if $X$ is non-spin and the complement $X-X^c$ of the fixed locus $X^c$ is simply-connected. By Proposition \ref{prop:admissible_pair}, a strongly admissible pair $(X,c)$ satisfies the hypotheses of Theorem \ref{thm:main}. 

\begin{proof}[Proof of Corollaries \ref{cor:ordinary_SW_formulation} and \ref{cor:symplectic}]
    If $(X,c)$ satisfies the hypotheses of Corollary \ref{cor:ordinary_SW_formulation}, then the pair $(X,c)$ is a strongly admissible pair satisfying condition \ref{item:3}. If $(X,\omega,c)$ satisfies the hypotheses of Corollary \ref{cor:symplectic}, then $(X,c)$ is a strongly admissible pair satisfying condition \ref{item:1}. 
\end{proof}

Finally, we prove Theorem \ref{thm:connected_sum}. The overall strategy is similar to the proofs of Theorems \ref{thm:main} and \ref{thm:isotopy}, but there are two notable adjustments. In the first part, additional care is needed to extend the homeomorphisms to the connected sum. In the second part, since $\pi_1(X-X^c)$ need not be trivial in this case, double covers of $Q$ branched along $B$ are not uniquely determined.

\begin{proof}[Proof of Theorem \ref{thm:connected_sum}]
    Note that $b_1(X)^{-c}=\sum_{i=1}^k b_1(X_i)^{-c_i}=0$. We may assume that $(X_1,c_1)$ is strongly admissible. For each $i=1,\dots,k$, there is a Real spin$^c$ structure $\mathfrak{s}_i$ on $(X_i,c_i)$ with $b_+(X_i)^{-c_i}=d(\mathfrak{s}_i)\in 2\mathbb{Z}$ and $|\deg _R(X_i,c_i,\mathfrak{s}_i)|>0$ by Proposition \ref{prop:admissible_pair}. Letting $\mathfrak{s}:=\mathfrak{s}_1\#\cdots \#\mathfrak{s}_k$, Theorem \ref{thm:connected_sum_formula} implies \[
    |\deg_R(X,c,\mathfrak{s})|=\prod_{i=1}^k |\deg_R(X_i,c_i,\mathfrak{s}_i)|>0,
    \]
    and in particular, $m(X,c)>0$. 

    For each $n\geq 1$, set \[
    (Y_n,\hat{d}_n):=\left((X_1,c_1)\#n(\Sigma,\iota)\right)\# (X_2,c_2)\#\cdots \#(X_k,c_k).
    \]
    We assume that the balls in $X_1$ used for connected sums with the other $X_i$'s are disjoint from the balls used for the connected sums with $\Sigma$'s. We have $X_1\#n\Sigma \cong X_1$ by Proposition \ref{prop:Gluck_twist}, so there exists an orientation-preserving diffeomorphism $g_n:Y_n\to X$. Set $d_n:=g_n\hat{d}_ng_n^{-1}:X\to X$. Theorem \ref{thm:topological_equivalence} gives an orientation-preserving homeomorphism $\Phi_n:X_1\#n\Sigma \to X_1$ such that $\Phi_n (c_1\# n\iota)=c_1\Phi_n$. By Remark \ref{rmk:small_ball}, we may assume that $\Phi_n$ is the identity on the balls in $X_1$ used for connected sums with the other $X_i$'s. It follows that $\Phi_n$ extends to a homeomorphism $\Psi_n:Y_n\to X$ such that $\Psi_n \hat{d}_n=c\Psi_n$. Now we have $(\Psi_n g_n^{-1})d_n=c(\Psi_ng_n^{-1})$, proving that the involutions $d_n$ are all topologically equivalent to $c$. On the other hand, we have $m(X,d_n)=3^n \cdot m(X,c)\neq 0$ for each $n\geq 1$ by Corollary \ref{cor:equivariant_connected_sum_degree}. Letting $d_0=c$, we see that $d_n$ and $d_m$ are smoothly inequivalent for $n\neq m$. Hence $\{d_n\}_{n=0}^\infty$ is an infinite exotic family of involutions on $X$ containing $c$.

    Assume now that all the fixed surfaces $X_i^{c_i}$ are connected. Set $Q_i:=X_i/\langle c_i\rangle$ and let $B_i$ be the image of $X_i^{c_i}$ in $Q_i$. Let $(Q',B')=(Q_2,B_2)\#\cdots \# (Q_k,B_k)$ so that $(Q,B)=(Q_1,B_1)\#(Q',B')$, and let $(E,E\cap B_1)\subset (Q_1,B_1)$ be a closed ball-disk pair used for the connected sum with $(Q',B')$. Since $(X_1,c_1)$ is strongly admissible, the pair $(Q_1,B_1)$ satisfies the hypotheses of Corollary \ref{cor:isotopy_general}. Thus, by taking connected sums with $\rho P$ away from $E$, we see that $B_1\# n(\rho P)$ and $B_1$ are topologically isotopic in $Q_1$ by an isotopy rel $E$, for each $n\geq 1$. Extending the isotopy by the identity shows that $\tilde{B}_n:=(B_1\#n(\rho P))\#B_2\# \cdots \#B_k$ and $\tilde{B}_0:=B_1\#\cdots \#B_k=B$ are topologically isotopic in $Q$ for each $n\geq 1$. On the other hand, since $H^1(Q-\tilde{B}_n;\mathbb{Z}_2)$ is finite, there are only finitely many orientation-preserving equivariant diffeomorphism types of double covers of $Q$ branched along $\tilde{B}_n$. Therefore, if the pairs $(Q,\tilde{B}_n)$ represent only finitely many orientation-preserving diffeomorphism types of pairs, then the pairs $(X,d_n)$ also represent only finitely many orientation-preserving equivariant diffeomorphism types, which is a contradiction. Hence we can find a sequence $0=n_0<n_1<n_2<\cdots$ so that $(Q,\tilde{B}_{n_j})$ are pairwise smoothly inequivalent. We conclude that $\{\tilde{B}_{n_j}\}_{j=0}^\infty$ is an infinite exotic family of surfaces in $Q$ containing $B$. 
\end{proof}

\section{Cyclic branched covers of \texorpdfstring{$\mathbb{CP}^2$}{CP2}}\label{sec:branched_covers}

In this section, we study cyclic branched covers of $\mathbb{CP}^2$ and construct anti-holomorphic involutions whose fixed loci have simply-connected complements. We then prove Theorem \ref{thm:branched_cover}. 
\subsection{Set up} Here we introduce the cyclic branched covers $X_{r,m}$ of $\mathbb{CP}^2$ and prove some general results.
We begin with the following elementary observation.
\begin{lemma}\label{lem:simply_connected_branched_cover}
    Let $M$ be a connected, smooth $4$-manifold and let $B\subset M$ be a connected, smoothly embedded closed surface. Suppose that $\pi_1(M-B)$ is cyclic, generated by a meridian $\mu$ of $B$. Let $p:\widetilde{M}\to M$ be a connected, smooth, cyclic branched cover of degree $r\geq 2$ with branch locus $B$ and ramification index $r$ along $p^{-1}(B)$. Then $\widetilde{M}$ is simply-connected.
\end{lemma}
\begin{proof}
    Note that $p_*\pi_1(\widetilde{M}-p^{-1}(B))=\langle \mu^r\rangle$. A meridian $\widetilde{\mu}$ of $p^{-1}(B)$ projects to $\mu^r$ since we have $p(z,u)=(z,u^r)$ in local coordinates. Thus, $\widetilde{\mu}$ generates $\pi_1(\widetilde{M}-p^{-1}(B))$. By transversality, the inclusion $\widetilde{M}-p^{-1}(B)\hookrightarrow \widetilde{M}$ induces a surjection on $\pi_1$. Since $\widetilde{\mu}$ is nullhomotopic in $\widetilde{M}$, we conclude that $\pi_1(\widetilde{M})=1$.
\end{proof}

Now, for integers $r\geq 2$ and $m\geq 1$, choose a section $F\in H^0(\mathbb{CP}^2,\mathcal{O}_{\mathbb{CP}^2}(rm))$ whose zero locus $B\subset \mathbb{CP}^2$ is a nonsingular curve of degree $rm$. Define \[
X_{r,m}:=\left\{(x,w)\in \mathcal{O}_{\mathbb{CP}^2}(m):w^{\otimes r}=F(x) \in \mathcal{O}_{\mathbb{CP}^2}(rm)\right\}.
\]
We suppress the dependence on $F$ in the notation; see part (a) of Proposition \ref{prop:properties_of_X_r,m} below. $X_{r,m}$ is a smooth projective surface, and the projection $\pi:X_{r,m}\to \mathbb{CP}^2$, $\pi(x,w)=x$, gives an $r$-fold cyclic cover branched along $B$ \cite[Section 7.1]{Gompf-Stipsicz-1999}. We first collect the basic properties of the surface $X_{r,m}$. Define \begin{equation}\label{eq:k_and_chi}
    k_{r,m}:=(r-1)m-3 \quad \text{and} \quad \chi_{r,m}:=\sum_{k=0}^{r-1}\frac{(km-1)(km-2)}{2}=r+\frac{r(r-1)m((2r-1)m-9)}{12}.
\end{equation}
\begin{proposition}\label{prop:properties_of_X_r,m}
    The complex surface $X_{r,m}$ has the following properties. 
    \begin{enumerate}[label=\textup{(\alph*)}, leftmargin=2em]
    \item Its oriented diffeomorphism type is independent of the choice of $F$. Moreover, the diffeomorphisms can be chosen to identify the class $H:=\pi^*c_1(\mathcal{O}_{\mathbb{CP}^2}(1))$. 
    \item It is simply-connected, and its topological Euler characteristic and the signature are given by \[
    e(X_{r,m})=3r+(r-1)rm(rm-3) \quad \text{and} \quad \sigma(X_{r,m})=r-\frac{r(r^2-1)m^2}{3}.
    \]
    In particular, $b_+(X_{r,m})=2\chi_{r,m}-1$. 
    \item It is spin if and only if $r$ is even and $m$ is odd.
    \item It is a minimal surface unless $(r,m)\in \{(2,2),(3,1)\}$, and it is of general type unless $(r,m)\in \{(2,1),(2,2),(2,3),(3,1),(4,1)\}$. 
\end{enumerate}
\end{proposition}
\begin{proof}
For part (a), consider two choices $F_0$ and $F_1$. Choose a smooth path $F_t$ ($t\in [0,1]$) so that the zero locus of $F_t$ is nonsingular for each $t$. Then the projection \[
\mathcal{X}:=\left\{(x,w,t)\in \mathcal{O}_{\mathbb{CP}^2}(m)\times [0,1]:w^{\otimes r}=F_t(x)\in \mathcal{O}_{\mathbb{CP}^2}(rm) \right\} \longrightarrow [0,1]
\]
is a smooth fiber bundle which must be trivial since $[0,1]$ is contractible. The resulting diffeomorphism between the fibers over $0$ and $1$ identifies the class $H$ since the class $P^*c_1(\mathcal{O}_{\mathbb{CP}^2}(1))$, where $P:\mathcal{X}\to \mathbb{CP}^2$ is the projection, restricts to the class $H$ on each fiber.

Simple connectivity follows from Lemma \ref{lem:simply_connected_branched_cover} and the well-known fact that $\pi_1(\mathbb{CP}^2-B)$ is cyclic of order $rm$, generated by a meridian $\mu$ of $B$. 

Next, we claim that the class $H\in H^2(X_{r,m};\mathbb{Z})$ is primitive. By part (a), we may use the particular nonsingular polynomial $F(z_0,z_1,z_2)=z_0^{rm}+z_1z_2^{rm-1}+z_1^{rm}$. On the line $L=\{z_1=0\}$, we have $F|_L=z_0^{rm}$, which shows that the image of the section $z_0^m$ of $\mathcal{O}_{\mathbb{CP}^2}(m)|_L$ is a holomorphically embedded sphere $S\subset X_{r,m}$ and $\pi|_S:S\to L$ is an isomorphism. Hence $\langle H,[S]\rangle=\langle c_1(\mathcal{O}_{\mathbb{CP}^2}(1)),[L]\rangle=1$; this proves the primitivity of $H$.

Let $\widetilde{B}\subset X_{r,m}$ denote the reduced ramification divisor. The tautological section of $\pi^*\mathcal{O}_{\mathbb{CP}^2}(m)$ vanishes simply along $\widetilde{B}$, so $\mathcal{O}_{X_{r,m}}(\widetilde{B})\cong \pi^*\mathcal{O}_{\mathbb{CP}^2}(m)$, and hence \[
K_{X_{r,m}}\cong \pi^*K_{\mathbb{CP}^2}\otimes \mathcal{O}_{X_{r,m}}((r-1)\widetilde{B})\cong \pi^*\mathcal{O}_{\mathbb{CP}^2}(-3+(r-1)m)=\pi^*\mathcal{O}_{\mathbb{CP}^2}(k_{r,m}).
\]
In particular, $K_{X_{r,m}}^2=rk_{r,m}^2$. On the other hand, since the genus of the curve $B$ is  $(rm-1)(rm-2)/2$, the topological Euler characteristic of $X_{r,m}$ is given by \[
e(X_{r,m})=r(e(\mathbb{CP}^2)-e(B))+e(B)=3r+(r-1)rm(rm-3).
\]
The formula $K_{X_{r,m}}^2=2e(X_{r,m})+3\sigma (X_{r,m})$ now yields the formula for $\sigma(X_{r,m})$. Since $b_1(X_{r,m})=0$, it follows that $b_+(X_{r,m})=(e(X_{r,m})-2+\sigma(X_{r,m}))/2=2\chi_{r,m}-1$, proving part (b). Part (c) follows from the fact that $w_2(X_{r,m})\equiv c_1(X_{r,m})\equiv k_{r,m}H\pmod{2}$, using the primitivity of $H$.

Finally, for part (d), note that if $k_{r,m}>0$, then $K_{X_{r,m}}$ is ample, so $X_{r,m}$ is a minimal surface of general type. If $k_{r,m}=0$, i.e., $(r,m)\in \{(2,3),(4,1)\}$, then $X_{r,m}$ is a K3 surface. If $k_{r,m}<0$, then $(r,m)$ is $(2,1)$, $(2,2)$, or $(3,1)$. The surface $X_{2,1}$ is isomorphic to $\mathbb{CP}^1\times \mathbb{CP}^1$, and $X_{2,2}$, $X_{3,1}$ are del Pezzo surfaces of degrees $2$ and $3$, respectively. This completes the proof.         
\end{proof}

\begin{remark}\label{rmk:chi_r,m}
    The integer $\chi_{r,m}$ is the holomorphic Euler characteristic of $X_{r,m}$, since $\chi(\mathcal{O}_{X_{r,m}})=(\sigma(X_{r,m})+e(X_{r,m}))/4=\chi_{r,m}$. It is easy to see that $\chi_{r,m}$ is even if and only if one of the following three conditions holds: \[
    \begin{cases}
        r\equiv0 \pmod{4}, \\
        r\equiv2 \pmod{4} \quad \text{and} \quad m\equiv 0,3\pmod{4}, \\
        r\equiv 3\pmod{4} \quad \text{and} \quad m\equiv 2,3\pmod{4}.
    \end{cases}
    \]
\end{remark}
\subsection{Fundamental group computation}The aim of this subsection is to prove Proposition \ref{prop:special_polynomial_F}.
From now on, we assume that $m$ is even; write $rm=2d$. Set $Q=z_0^2+z_1^2+z_2^2$; then $C:=\{Q=0\}$ is a nonsingular conic in $\mathbb{CP}^2$. Choose $2d$ pairwise non-proportional linear polynomials $\ell_1,\dots,\ell_{2d}\in H^0(\mathbb{CP}^2,\mathcal{O}_{\mathbb{CP}^2}(1))$ with real coefficients; for example, we may take $\ell_j=z_0+j z_1+j^2z_2$ ($1\leq j\leq 2d$). Now for $\epsilon>0$, set \begin{equation}\label{eq:special_polynomial_F}
    F=Q^d+ \epsilon^d \ell_1\cdots \ell_{2d}\in H^0(\mathbb{CP}^2,\mathcal{O}_{\mathbb{CP}^2}(2d)).
\end{equation}
This is a real homogeneous polynomial of degree $2d$. 

\begin{proposition}\label{prop:special_polynomial_F}
    For sufficiently small $\epsilon>0$, the polynomial $F$ satisfies the following conditions. \begin{enumerate}[label=\textup{(\alph*)}, leftmargin=2em]
    \item The zero locus $B$ of $F$ is nonsingular.
    \item $F(x)>0$ for all $x\in \mathbb{R}^3-\{0\}$, so that $B\cap \mathbb{RP}^2=\varnothing$.
    \item $\pi_1(\mathbb{CP}^2-(\mathbb{RP}^2\cup B))\cong \mathbb{Z}_{4d}$, generated by a meridian of $B$. 
    \end{enumerate}
\end{proposition}
\begin{proof}
 Since $Q\equiv 1$ on the unit sphere $S^2\subset \mathbb{R}^3$, we have $F(x)>0$ for all $x\in S^2$ for small $\epsilon>0$. By homogeneity, it follows that $F(x)>0$ for all $x\in \mathbb{R}^3-\{0\}$, proving part (b). 

Write $P:=\ell_1\cdots \ell_{2d}$. Each line $\{\ell_j=0\}$ intersects $C$ transversely in two points, and for $j\neq k$, the lines $\{\ell_j=0\}$ and $\{\ell_k=0\}$ intersect at a real point. Consequently, there are $4d$ zeros of the section $P|_C\in H^0(C,\mathcal{O}_{\mathbb{CP}^2}(2d)|_C)$, which are all simple. Define \[
D:=\{(p,v)\in \mathcal{O}_{\mathbb{CP}^2}(2)|_C:v^{\otimes d}+P(p)=0\}.
\]
It is a smooth irreducible compact curve in $\mathcal{O}_{\mathbb{CP}^2}(2)|_C$, and the projection $D\to C$ is a cyclic cover of degree $d$, branched along the $4d$ zeros of $P|_C$. We will prove the following lemma below.

\begin{lemma}\label{lem:fundamental_group_computation}
    $\pi_1(\mathcal{O}_{\mathbb{CP}^2}(2)|_C-D)\cong \mathbb{Z}_{4d}$, generated by a meridian of $D$.  
\end{lemma}

The tautological bundle $\mathcal{O}_{\mathbb{CP}^2}(-1)\subset \mathbb{CP}^2\times \mathbb{C}^3$ has a Hermitian metric induced by the standard metric on $\mathbb{C}^3$. By taking dual and tensor powers, we obtain a Hermitian metric on $\mathcal{O}_{\mathbb{CP}^2}(n)$ for each $n\geq 1$. We denote the induced fiberwise norm on $\mathcal{O}_{\mathbb{CP}^2}(n)$ by $\|\cdot\|_n$. If $G$ is a homogeneous polynomial of degree $n$, then \[
\|G([z])\|_n^2=\frac{|G(z_0,z_1,z_2)|^2}{(|z_0|^2+|z_1|^2+|z_2|^2)^n} \quad ([z]=[z_0:z_1:z_2]\in \mathbb{CP}^2),
\] 
and $\|v^{\otimes k}\|_{kn}=\|v\|^k_n$ for any $v\in \mathcal{O}_{\mathbb{CP}^2}(n)$ and $k\geq 1$. 

Choose $R>0$ large enough so that $R^d>\max_{x\in \mathbb{CP}^2} \|P(x)\|_{2d}$, and define \begin{align*}
    \mathcal{W}&:=\left\{(x,v,t)\in \mathcal{O}_{\mathbb{CP}^2}(2)\times \mathbb{R}:\|v\|_{2}\leq R, \,Q(x)=tv\in \mathcal{O}_{\mathbb{CP}^2}(2) \right\}, \\
    \mathcal{A}&:=\left\{(x,v,t)\in \mathcal{W}: v^{\otimes d}+P(x)=0\in \mathcal{O}_{\mathbb{CP}^2}(2d) \right\}.
\end{align*}
Note that $\|v\|_2<R$ for all $(x,v,t)\in \mathcal{A}$ by our choice of $R$. For $t\in \mathbb{R}$, let $W_t$ denote the fiber of the projection $\mathcal{W}\to \mathbb{R}$ over $t$, and let $A_t:=W_t\cap \mathcal{A}$. The following lemma will be proved below.
\begin{lemma}\label{lem:diffeomorphism_of_pairs}
  For all sufficiently small $|t|$, $W_t$ is a compact smooth $4$-manifold with boundary, $A_t\subset W_t-\partial W_t\subset \mathcal{O}_{\mathbb{CP}^2}(2)$ is a nonsingular compact complex curve, and there is a diffeomorphism $(W_t,A_t)\cong (W_0,A_0)$ of pairs.   
\end{lemma}

Observe that $W_0$ is the disk subbundle of $\mathcal{O}_{\mathbb{CP}^2}(2)|_C$ of radius $R$, and $A_0=D$. Thus, $\mathcal{O}_{\mathbb{CP}^2}(2)|_C-D$ deformation retracts onto $W_0-A_0=W_0-D$. On the other hand, for sufficiently small $\epsilon$, restricting the projection $\mathcal{O}_{\mathbb{CP}^2}(2)\to \mathbb{CP}^2$ gives an identification $(W_\epsilon,A_\epsilon)\cong (T,B)$, where $T:=\{x\in \mathbb{CP}^2:\|Q(x)\|_2\leq \epsilon R\}$. In particular, $B$ is nonsingular, proving part (a). Now consider the map $\rho:\mathbb{CP}^2\to \mathbb{R}$ defined by \[
\rho([z])=\|Q([z])\|_2 =\frac{|z_0^2+z_1^2+z_2^2|}{|z_0|^2+|z_1|^2+|z_2|^2} \quad ([z]=[z_0:z_1:z_2]\in \mathbb{CP}^2). 
\]
We have $0\leq \rho\leq 1$, $\rho^{-1}(0)=C$, $\rho^{-1}(1)=\mathbb{RP}^2$, and $\rho$ has no critical points on $\rho^{-1}(0,1)$. Therefore, a standard argument using the flow of $\nabla \rho /\|\nabla \rho\|^2$ (cf. the proof of \cite[Theorem 3.1]{Milnor-1963}) shows that $\rho^{-1}[0,1)$ deformation retracts onto $T=\rho^{-1}[0,\epsilon R]$, assuming $\epsilon>0$ is sufficiently small so that $\epsilon R<1$. This restricts to a deformation retraction of $\rho^{-1}([0,1))-B=\mathbb{CP}^2-\mathbb{RP}^2-B$ onto $T-B$, since $B=\{F=0\}\subset \operatorname{int}T$. In summary, we have the following diagram: \[
\begin{tikzcd}
    \mathbb{CP}^2-\mathbb{RP}^2-B & T-B\cong W_\epsilon-A_\epsilon \ar[l, hook', "\simeq"'] \ar[r, "\cong"] & W_0-A_0=W_0-D \ar[r, hook, "\simeq"] & \mathcal{O}_{\mathbb{CP}^2}(2)|_C-D.
\end{tikzcd}
\]
The resulting isomorphism $\pi_1(\mathbb{CP}^2-\mathbb{RP}^2-B)\to \pi_1(\mathcal{O}_{\mathbb{CP}^2}(2)|_C-D)$ sends the class of a meridian of $B$ to the class of a meridian of $D$. Consequently, part (c) follows from Lemma \ref{lem:fundamental_group_computation}. 
\end{proof} 

It remains to prove Lemmas \ref{lem:fundamental_group_computation} and \ref{lem:diffeomorphism_of_pairs}. For Lemma \ref{lem:fundamental_group_computation}, we will make use of the following result, which is a special case of \cite[Corollary 2.5]{Nori-1983}. 
\begin{proposition}[{\cite{Nori-1983}}]\label{prop:Nori}
    Let $X$ be a smooth connected projective complex surface, and let $D, E\subset X$ be transversally intersecting smooth irreducible complex curves with $[D]^2>0$. Then the kernel of the homomorphism $\pi_1(X-(D\cup E))\to \pi_1(X-E)$ induced by inclusion is contained in the center of $\pi_1(X-(D\cup E))$. 
\end{proposition}

\begin{proof}[Proof of Lemma \ref{lem:fundamental_group_computation}]
    Since $N:=\mathcal{O}_{\mathbb{CP}^2}(2)|_C\cong \mathcal{O}_{\mathbb{CP}^1}(4)$, we may compactify $N$ to the Hirzebruch surface $\mathbb{F}_4$ by adding the infinity section $S_\infty$ with $[S_\infty]^2=-4$. The zero section $C\subset N$ satisfies $[C]^2=4$, and the projection $N\to C$ is a homotopy equivalence. Since the restriction $D\to C$ has degree $d$, it follows that $[D]=d[C]\in H_2(N;\mathbb{Z})$; in particular, $[D]^2=4d^2>0$. Applying Proposition \ref{prop:Nori} shows that the kernel of the homomorphism $\pi_1(\mathbb{F}_4-D-S_\infty)\to \pi_1(\mathbb{F}_4-S_\infty)$ is central. Since $\pi_1(\mathbb{F}_4-S_\infty)=\pi_1(N)=1$, it follows that $\pi_1(\mathbb{F}_4-D-S_\infty)=\pi_1(N-D)$ is abelian. 

    Next, consider the following portion of the homology long exact sequence of the pair $(N,N-D)$: \[
    H_2(N;\mathbb{Z})\xrightarrow{j_*} H_2(N,N-D;\mathbb{Z}) \xrightarrow{\partial} H_1(N-D;\mathbb{Z}) \longrightarrow H_1(N;\mathbb{Z})=0.
    \]
    The group $H_2(N;\mathbb{Z})$ is infinite cyclic generated by the class $[C]$. If $\nu D$ denotes a closed tubular neighborhood of $D$ in $N$, then $H_2(N,N-D;\mathbb{Z})\cong H_2(\nu D,\partial \nu D;\mathbb{Z})\cong H_0(D;\mathbb{Z})$, where the second isomorphism is the Thom isomorphism. Thus, $H_2(N,N-D;\mathbb{Z})$ is infinite cyclic generated by the class of a positively oriented normal disk $\Delta$. Since $[C]\cdot [D]=d[C]^2=4d$, we have $j_*[C]=4d[\Delta]$. It follows that $H_1(N-D;\mathbb{Z})$ is cyclic of order $4d$, generated by $\partial[\Delta]=[\partial \Delta]$, which is the class of a positively oriented meridian of $D$. Since $\pi_1(N-D)$ is abelian, this completes the proof of Lemma \ref{lem:fundamental_group_computation}.     
\end{proof}

For Lemma \ref{lem:diffeomorphism_of_pairs}, we use the following relative form of Ehresmann's fibration theorem. 
\begin{lemma}
    \label{lem:Ehresmann}
    Let $W$ be a smooth manifold with boundary, and let $A\subset W-\partial W$ be a smoothly embedded submanifold without boundary, closed as a subset of $W$. For an open interval $I\subset \mathbb{R}$, suppose that $f:W\to I$ is a proper surjective smooth map and that $f$, $f|_{\partial W}$, and $f|_A$ are submersions. Then the pairs $(f^{-1}(t), f^{-1}(t)\cap A)$ are diffeomorphic for all $t\in I$.
\end{lemma}
\begin{proof}
    By the submersion hypotheses, the vector field $\frac{\partial}{\partial t}$ on $I$ has local lifts to $W$ that are tangent to $A$ on $A$ and tangent to $\partial W$ on $\partial W$. Since $A$ and $\partial W$ are disjoint closed subsets of $W$, these local lifts can be patched by a partition of unity to a smooth vector field $V$ on $W$ satisfying $df(V)=1$, which is tangent to $A$ on $A$ and tangent to $\partial W$ on $\partial W$. Its local flow $\Phi_s$  preserves $A$ and $\partial W$. Moreover, $f(\Phi_s(q))=f(q)+s$ whenever $\Phi_s(q)$ is defined. Now, for any $t_0,t_1\in I$ with $t_0<t_1$, properness of $f$ implies that $f^{-1}[t_0,t_1]$ is compact, so $\Phi_{t_1-t_0}$ is a well-defined diffeomorphism \[
    (f^{-1}(t_0), f^{-1}(t_0)\cap A)\longrightarrow (f^{-1}(t_1), f^{-1}(t_1)\cap A)
    \]
    of pairs, with inverse $\Phi_{t_0-t_1}$.
\end{proof}

\begin{proof}[Proof of Lemma \ref{lem:diffeomorphism_of_pairs}]
    Consider the space \[
    \widetilde{\mathcal{W}}:=\left\{(x,v,t)\in \mathcal{O}_{\mathbb{CP}^2}(2)\times \mathbb{R}: Q(x)=tv\in \mathcal{O}_{\mathbb{CP}^2}(2) \right\},
    \]
    so that $\mathcal{W}=\widetilde{\mathcal{W}}\cap \{\|v\|_2\leq R\}$. Let $f:\widetilde{\mathcal{W}}\to \mathbb{R}$ denote the projection. 

    \vspace{1mm}
    \noindent \textbf{Step 1:} Near any point in $\widetilde{\mathcal{W}}\cap \{t=0\}$, $\widetilde{\mathcal{W}}$ is a smooth $5$-manifold and the projection $f:\widetilde{\mathcal{W}}\to \mathbb{R}$ is a submersion.

    \begin{proof}
        Fix a point $(x,v,0)\in \widetilde{\mathcal{W}}$ so that $x\in C$. Near $x$, choose a holomorphic frame $e$ of $\mathcal{O}_{\mathbb{CP}^2}(2)$. Since $Q\in H^0(\mathbb{CP}^2,\mathcal{O}_{\mathbb{CP}^2}(2))$ is transverse to the zero section along $C$, we can choose holomorphic coordinates $(z,y)$ of $\mathbb{CP}^2$ near $x$ so that $Q=ye$ and $C=\{y=0\}$. Writing $v=ue$, the equation $Q(x)=tv$ becomes $y=tu$. Thus, near $(x,v,0)$, $\widetilde{\mathcal{W}}$ is parametrized by $(z,u,t)\in \mathbb{C}\times \mathbb{C}\times \mathbb{R}\mapsto ((z,tu), ue_{(z,tu)},t)$. In these coordinates, we have $f(z,u,t)=t$. This proves \textbf{Step 1}.
    \end{proof}

    \noindent \textbf{Step 2:} For any point in $\widetilde{\mathcal{W}}\cap \{t=0, \, \|v\|_2\leq R\}$, $\mathcal{W}$ is a smooth $5$-manifold with boundary $\mathcal{W}\cap \{\|v\|_2=R\}$ near that point.

    \begin{proof}
    Using the notation in \textbf{Step 1}, set $h(z,y):=\|e_{(z,y)}\|_2^2$ and $b(z,u,t):=h(z,tu)|u|^2-R^2$, so that $\mathcal{W}=\widetilde{\mathcal{W}}\cap \{b\leq 0\}$ near $(x,v,0)$. At a point with $b(z,u,0)=0$, writing $u=u_1+iu_2$, the differential of $b$ satisfies \[
    (db)_{(z,u,0)}\left(u_1\frac{\partial}{\partial u_1}+u_2\frac{\partial}{\partial u_2} \right)=2h(z,0)|u|^2=2\|ue_{(z,0)}\|_2^2=2R^2>0.
    \]
    This shows that $\mathcal{W}=\{b\leq 0\}$ is a smooth $5$-manifold with boundary $\{b=0\}$ near $(z,u,0)$. 
    \end{proof}

    \noindent \textbf{Step 3:} For any point in $\widetilde{\mathcal{W}}\cap \{t=0, \, \|v\|_2=R\}$,  the restriction $f:\partial \mathcal{W}\to \mathbb{R}$ is a submersion near that point.
    \begin{proof}
        Using the notation in \textbf{Step 2}, consider a point with $b(z,u,0)=0$. Set \[ 
        V_u:=u_1\frac{\partial}{\partial u_1}+u_2\frac{\partial}{\partial u_2} \quad \text{and} \quad V_t:=\frac{\partial}{\partial t}-\frac{\partial b}{\partial t} \frac{1}{(db)(V_u)} V_u.
    \]
    Since $(db)_{(z,u,0)}(V_u)>0$ by the proof of \textbf{Step 2}, $V_t$ is well-defined near $(z,u,0)$. Also, $V_t$ is tangent to $\partial \mathcal{W}$ since $(db)(V_t)=0$. The proof of \textbf{Step 1} shows that the projection $f:\widetilde{\mathcal{W}}\to \mathbb{R}$ is locally given by $(z,u,t)\mapsto t$, so $(df)(V_u)=0$ and $(df)(V_t)=1$. Thus, the differential of $f|_{\partial \mathcal{W}}$ is nonzero near $(z,u,0)$. 
    \end{proof}

    \noindent \textbf{Step 4:} For any point in $A_0=\mathcal{A}\cap \{t=0\}$, $\mathcal{A}$ is a smooth submanifold of $\mathcal{W}-\partial \mathcal{W}$ and the restriction $f:\mathcal{A}\to \mathbb{R}$ is a submersion near that point. 
    \begin{proof}
        We have already seen that $\|v\|_2<R$ for all $(x,v,t)\in \mathcal{A}$, so $\mathcal{A}$ is contained in $\mathcal{W}-\partial \mathcal{W}$. Using the notation of the previous steps, write $P(z,y)=p(z,y)e_{(z,y)}^{\otimes d}$, so that $\mathcal{A}=\mathcal{W}\cap \{G=0\}$ where $G(z,u,t):=u^d+p(z,tu)$. For a fixed $t$, $G$ is holomorphic in $z$ and $u$. Also, if $p(z,0)=0$ then $\frac{\partial}{\partial z}p(z,0)\neq 0$ since $P|_C$ has only simple zeros. Thus, if $G(z,u,0)=0$, then at least one of $\frac{\partial G}{\partial z}$ and $\frac{\partial G}{\partial u}$ must be nonzero at $(z,u,0)$. This proves that $\mathcal{A}$ is a smooth submanifold of $\mathcal{W}-\partial \mathcal{W}$ near $(z,u,0)$. 

        The above paragraph shows that the differential $\mathbb{C}^2\to \mathbb{C}$ of the map $(z,u)\mapsto G(z,u,0)$ is surjective at any point with $G(z,u,0)=0$. Thus, for any $\lambda\in \mathbb{R}$, there exist $\xi,\eta\in \mathbb{C}$ such that \[
        \xi \frac{\partial G}{\partial z}+\eta\frac{\partial G}{\partial u}=-\lambda \frac{\partial G}{\partial t} \quad \text{at $(z,u,0)$}.
        \]
        This shows that $(\xi,\eta,\lambda)\in \mathbb{C}^2\times \mathbb{R}$, viewed as a tangent vector to $\mathcal{W}$ at $(z,u,0)$, is tangent to $\mathcal{A}$. Since $(df)_{(z,u,0)}(\xi,\eta,\lambda)=\lambda$, it follows that $f:\mathcal{A}\to \mathbb{R}$ is a submersion near $(z,u,0)$.  
    \end{proof}

     Finally, we show that the preceding assertions hold for sufficiently small $|t|$:

    \vspace{1mm}
    \noindent \textbf{Step 5:} There exists $\tau>0$ such that the following statements hold for $|t|<\tau$. \begin{itemize}[leftmargin=2em]
        \item $\mathcal{W}$ is a smooth $5$-manifold with boundary $\mathcal{W}\cap \{\|v\|_2= R\}$. 
        \item $\mathcal{A}$ is a smooth submanifold of $\mathcal{W}-\partial \mathcal{W}$.
        \item The projections $f:\mathcal{W}\to (-\tau,\tau)$, $f:\partial \mathcal{W}\to (-\tau,\tau)$, and $f:\mathcal{A}\to (-\tau,\tau)$ are submersions. 
        \item $A_t=f^{-1}(t)\cap \mathcal{A}$ is a nonsingular compact complex curve.  
    \end{itemize}
    \begin{proof}
        For the first condition, suppose that there is a sequence $(x_j,v_j,t_j)$ in $\mathcal{W}$ with $t_j\to 0$ such that $\mathcal{W}$ is not a smooth manifold with boundary near each $(x_j,v_j,t_j)$. The sequence $(x_j,v_j)$ is contained in the compact disk bundle in $\mathcal{O}_{\mathbb{CP}^2}(2)$ of radius $R$, so we may assume that $(x_j,v_j)\to (x_\infty,v_\infty)$ by passing to a subsequence. Then $(x_j,v_j,t_j)\to (x_\infty,v_\infty,0)$ and $(x_\infty,v_\infty,0)\in \mathcal{W}$ since $\mathcal{W}$ is closed in $\mathcal{O}_{\mathbb{CP}^2}(2)\times \mathbb{R}$. By \textbf{Step 2}, $\mathcal{W}$ is locally a smooth $5$-manifold with boundary near $(x_\infty,v_\infty,0)$, which gives a contradiction since $(x_j,v_j,t_j)$ is contained in a neighborhood of $(x_\infty,v_\infty,0)$ for all large $j$. The second and third conditions can be verified in exactly the same way; note that being a submersion is an open condition. 

        The last condition is also similar. If we suppose, on the contrary, that there is a sequence $(x_j,v_j,t_j)$ in $\mathcal{A}$ such that $t_j\to 0$ and $(x_j,v_j,t_j)$ is a singular point of $A_{t_j}$, then by passing to a subsequence, we may assume that $(x_j,v_j,t_j)\to (x_\infty,v_\infty,0)\in A_0$. For a local chart near $(x_\infty,v_\infty,0)$ as in the proof of \textbf{Step 4}, we have seen that at least one of $\frac{\partial G}{\partial z}$ and $\frac{\partial G}{\partial u}$ is nonzero at $(x_\infty,v_\infty,0)$; thus it is also nonzero at $(x_j,v_j,t_j)$ for large $j$, but this contradicts that $A_{t_j}$ is singular at $(x_j,v_j,t_j)$. For each $t$, $z$ and $u$ are holomorphic coordinates on $W_t-\partial W_t$, and $(z,u)\mapsto G(z,u,t)$ is holomorphic, so we conclude that $A_t$ is a nonsingular complex curve in $W_t-\partial W_t$ for all small $|t|$. It is obvious that each $A_t$ is compact.
    \end{proof}

    It is clear that the projection $f:\mathcal{W}\to \mathbb{R}$ is surjective and proper. Thus, by \textbf{Step 5}, Lemma \ref{lem:Ehresmann} applies to $W=\mathcal{W}\cap \{|t|<\tau\}$, $A=\mathcal{A}\cap \{|t|<\tau\}$, and $I=(-\tau,\tau)$. This completes the proof of Lemma \ref{lem:diffeomorphism_of_pairs}. 
\end{proof}

\subsection{The involution} 
We now use the polynomial $F$ in \eqref{eq:special_polynomial_F} to construct an anti-holomorphic involution on $X_{r,m}$ whose fixed locus has a simply-connected complement. 

\begin{proposition}\label{prop:conjugation}
    Let $r,m\geq 2$ be integers with $m$ even. Write $rm=2d$ and choose a polynomial $F\in H^0(\mathbb{CP}^2,\mathcal{O}_{\mathbb{CP}^2}(2d))$ as in \eqref{eq:special_polynomial_F} for sufficiently small $\epsilon>0$. Then there exists a smooth involution $c:X_{r,m}\to X_{r,m}$ satisfying the following conditions. \begin{enumerate}[label=\textup{(\alph*)}, leftmargin=2em]
        \item The involution $c$ is anti-holomorphic and orientation-preserving. Its fixed locus is given by \[
        X_{r,m}^c\cong \begin{cases}
            \mathbb{RP}^2 & \text{if $r$ is odd}, \\
            \mathbb{RP}^2\amalg \mathbb{RP}^2 & \text{if $r$ is even}.
        \end{cases}
        \]
        \item There is a K\"ahler form $\omega$ on $X_{r,m}$ such that $[\omega]=H$, $c^*\omega=-\omega$, and $b_+(X_{r,m})^{-c}=\chi_{r,m}$. 
        \item The complement $X_{r,m}-X_{r,m}^c$ is simply-connected.
    \end{enumerate}
\end{proposition}
\begin{proof} The complex conjugation on $\mathbb{CP}^2$ has a standard anti-linear lift to $\mathcal{O}_{\mathbb{CP}^2}(m)$. Since $F$ has real coefficients, the lift restricts to an anti-holomorphic smooth involution $c:X_{r,m}\to X_{r,m}$; it must be orientation-preserving since $\dim_{\mathbb{C}}X_{r,m}=2$. 

Note that $Q^{m/2}$ defines a nowhere-vanishing section of $\mathcal{O}_{\mathbb{CP}^2}(m)$ over $\mathbb{RP}^2$. The function $f:=F/Q^{d}:\mathbb{RP}^2\to \mathbb{R}_{>0}$ is well-defined and satisfies $F(x)=f(x)(Q^{m/2}(x))^{\otimes r}\in \mathcal{O}_{\mathbb{CP}^2}(2d)$ for $x\in \mathbb{RP}^2$. Letting $\zeta=e^{2\pi i /r}$, the preimage $\pi^{-1}(\mathbb{RP}^2)$ is the disjoint union of the images $L_j$ ($j\in \mathbb{Z}_r$) of the sections $\mathbb{RP}^2\to \mathcal{O}_{\mathbb{CP}^2}(m)$, $x\mapsto \zeta^j f(x)^{1/r}Q^{m/2}(x)$. For each $j$, the restriction $\pi:L_j\to \mathbb{RP}^2$ is a diffeomorphism and $c(L_j)=L_{-j}$. It follows that the fixed locus $X_{r,m}^c$ is precisely $L_0$ if $r$ is odd, and $L_0\amalg L_{r/2}$ if $r$ is even. This proves part (a). 

Since the class $H\in H^2(X_{r,m};\mathbb{Z})$ is ample, it has a K\"ahler representative $\eta$. Since $c$ is anti-holomorphic, $-c^*\eta$ is again a K\"ahler form. Therefore, $\omega:=(\eta-c^*\eta)/2$ is a K\"ahler form on $X_{r,m}$ satisfying $c^*\omega=-\omega$ and $[\omega]=H$; note that $c^*H=-H$. The associated K\"ahler metric $g$ is $c$-invariant, and we have the following $c$-invariant decomposition of the $g$-self-dual real harmonic $2$-forms: \[
\mathcal{H}^+_g(X_{r,m};\mathbb{R}) =\mathbb{R}\omega \oplus \left( H^{2,0}(X_{r,m})\oplus H^{0,2}(X_{r,m})\right)_{\mathbb{R}}.
\]
For the anti-linear involution on $H^{2,0}(X_{r,m})$ given by $\alpha\mapsto \overline{c^*\alpha}$, its real fixed subspace has dimension $\dim_{\mathbb{C}} H^{2,0}(X_{r,m})=p_g(X_{r,m})$. If $\{\alpha_1,\dots,\alpha_{p_g}\}$ is a real basis for this real fixed subspace, then $\{\operatorname{Re}\alpha_1,\dots,\operatorname{Re}\alpha_{p_g},\operatorname{Im}\alpha_1,\dots,\operatorname{Im}\alpha_{p_g}\}$ is a real basis for $(H^{2,0}(X_{r,m})\oplus H^{0,2}(X_{r,m}))_{\mathbb{R}}$, and we have $c^*\operatorname{Re}\alpha_j=\operatorname{Re}\alpha_j$, $c^*\operatorname{Im}\alpha_j=-\operatorname{Im}\alpha_j$. We conclude that $b_+(X_{r,m})^{-c}=1+p_g(X_{r,m})=(b_+(X_{r,m})+1)/2=\chi_{r,m}$, proving part (b). 

Finally, observe that the restriction $\pi:\pi^{-1}(\mathbb{CP}^2-\mathbb{RP}^2)\to \mathbb{CP}^2-\mathbb{RP}^2$ is a cyclic cover of degree $r$ branched along the zero locus $B$ of $F$. By Proposition \ref{prop:special_polynomial_F}(c), $\pi_1(\mathbb{CP}^2-\mathbb{RP}^2-B)$ is cyclic of order $2rm$, generated by a meridian $\mu$ of $B$. Thus, $\pi^{-1}(\mathbb{CP}^2-\mathbb{RP}^2)$ is simply-connected by Lemma \ref{lem:simply_connected_branched_cover}. The proof of part (a) shows that $\pi^{-1}(\mathbb{CP}^2-\mathbb{RP}^2)$ is obtained from $X_{r,m}-X_{r,m}^c$ by removing the $L_j$'s which are not fixed by $c$. Hence the inclusion $\pi^{-1}(\mathbb{CP}^2-\mathbb{RP}^2)\hookrightarrow X_{r,m}-X_{r,m}^c$ induces a surjection on $\pi_1$ by transversality. This completes the proof.     
\end{proof}

We now combine the preceding propositions to prove Theorem \ref{thm:branched_cover}. 

\begin{proof}[Proof of Theorem \ref{thm:branched_cover}]
    Suppose both $m$ and $\chi_{r,m}$ are even. By Proposition \ref{prop:conjugation}, there is an anti-holomorphic involution $c$ on $X_{r,m}$ such that the fixed locus is \[
        X_{r,m}^c\cong \begin{cases}
            \mathbb{RP}^2 & \text{if $r$ is odd}, \\
            \mathbb{RP}^2\amalg \mathbb{RP}^2 & \text{if $r$ is even}.
        \end{cases}
        \] and $X_{r,m}-X_{r,m}^c$ is simply-connected. Moreover, there is a K\"ahler form $\omega$ on $X_{r,m}$ with $c^*\omega=-\omega$ and $[\omega]=H$. Now, Proposition \ref{prop:properties_of_X_r,m} shows that $b_+(X_{r,m})=2\chi_{r,m}-1\equiv 3\pmod{4}$ and that $X_{r,m}$ is a non-spin complex surface of general type. Consequently, $(X_{r,m},\omega,c)$ satisfies the hypotheses of Corollary \ref{cor:symplectic}, so Theorem \ref{thm:main} implies $X_{r,m}$ admits an infinite exotic family $\{c_n\}_{n=0}^\infty$ of involutions with $c_0=c$, such that every $c_n$ with $n\geq 1$ is not anti-symplectic with respect to any orientation-compatible symplectic form. The assertion about exotic surfaces in $Q_{r,m}$ follows from Theorem \ref{thm:isotopy}.

    Simple connectivity of $X_{r,m}$ is proved in Proposition \ref{prop:properties_of_X_r,m}(b). To prove irreduciblity of $X_{r,m}$, it suffices to prove symplectic minimality of $(X_{r,m},\omega)$; see \cite[Corollary 5.5]{Kotschick-1997}. (Note also that smooth minimality of $X_{r,m}$ is equivalent to symplectic minimality of $(X_{r,m},\omega)$ by \cite[Proposition 2.1]{Li-2006}.) The proofs of Propositions \ref{prop:properties_of_X_r,m} and \ref{prop:conjugation} (see also Remark \ref{rmk:chi_r,m}) show that \[
    K_\omega:=-c_1(X_{r,m},\omega)=c_1(K_{X_{r,m}})=k_{r,m}H=k_{r,m}[\omega]
    \] with $k_{r,m}>0$. If $E\subset (X_{r,m},\omega)$ is a symplectic $(-1)$-sphere, then the adjunction formula gives $K_\omega\cdot E=-1$, whereas $K_\omega\cdot E=k_{r,m}\int_E \omega>0$, which is a contradiction. This proves the symplectic minimality of $(X_{r,m},\omega)$. In fact, symplectic minimality of $(X_{r,m},\omega)$ also follows from \cite[Theorem 1.2]{Hamilton-Kotschick-2006}.
\end{proof}

\begin{remark}\label{rmk:cohomology_antisymplectic}
Although $c_n$ for each $n\geq 1$ is not anti-symplectic with respect to any orientation-compatible symplectic form on $X_{r,m}$, it satisfies $c_n^*[\omega]=-[\omega]$ on cohomology for the K\"ahler form $\omega$ used in the proof of Theorem \ref{thm:branched_cover}. To prove this, we use the fact that if $X$ is a minimal K\"ahler surface of general type, $\tilde{X}\to X$ is a blow-up at a point, and $f:\tilde{X}\to \tilde{X}$ is an orientation-preserving diffeomorphism, then $f^*\tilde{K}=\pm \tilde{K}$, where $\tilde{K}$ is the pullback of the canonical class $K\in H^2(X;\mathbb{Z})$ of $X$; this is a special case of \cite[Corollary 1.3]{Friedman-Morgan-1997}. 

Now recall from the proof of Theorem \ref{thm:main} that $c_n:X_{r,m}\to X_{r,m}$ is defined to be the composition $g_n\circ (c\# n\iota)\circ g_n^{-1}$ where $g_n$ is an orientation-preserving diffeomorphism $X_{r,m}\# n\Sigma\to X_{r,m}$. Let $K:=c_1(K_{X_{r,m}})=k_{r,m}[\omega]\in H^2(X_{r,m};\mathbb{Z})$, and let $h_n:n\Sigma\# \CPbar\to \CPbar$ be an orientation-preserving diffeomorphism; such an $h_n$ exists by Proposition \ref{prop:Gluck_twist}. If $\tilde{X}\cong X_{r,m}\# \CPbar\to X_{r,m}$ is a blow-up at a point, then the composition \[
\tilde{X} \cong X_{r,m}\# \CPbar \xrightarrow{(\operatorname{id}\#h_n)^{-1}} X_{r,m}\#n\Sigma \# \CPbar \xrightarrow{g_n\# \operatorname{id}} X_{r,m}\# \CPbar\cong \tilde{X}
\]
is an orientation-preserving diffeomorphism, so the above fact implies that $(g_n\#\operatorname{id})^*\tilde{K}=\pm(\operatorname{id}\#h_n)^*\tilde{K}$. This gives $g_n^*K=\pm K_n$, where $K_n$ is the class corresponding to $K$ under the natural identification $H^2(X_{r,m}\#n\Sigma;\mathbb{Z})\cong H^2(X_{r,m};\mathbb{Z})$. Since $c^*K=-K$ and $H^2(\Sigma;\mathbb{Z})=0$, we clearly have $(c\# n\iota)^*K_n=-K_n$, and therefore \[
c_n^*K=(g_n^{-1})^* (c\# n\iota)^*g_n^* K =(g_n^{-1})^* (c\# n\iota)^* (\pm K_n)=(g_n^*)^{-1}(\mp K_n)=-K.
\]
This proves that $c_n^*[\omega]=-[\omega]$, as desired. 
\end{remark}

We end this section by determining the oriented homeomorphism type of the quotient. For the involution $c$ on $X_{r,m}$ given in Proposition \ref{prop:conjugation}, let $Q_{r,m}:=X_{r,m}/\langle c\rangle$ denote the quotient $4$-manifold. Recall the integers $k_{r,m}$ and $\chi_{r,m}$ defined in \eqref{eq:k_and_chi}, and set \[
K_{r,m}:=rk_{r,m}^2=r\left((r-1)m-3\right)^2.
\]
We have $K_{r,m}=K_{X_{r,m}}^2$ and $\chi_{r,m}=\chi(\mathcal{O}_{X_{r,m}})$; see Proposition \ref{prop:properties_of_X_r,m} and Remark \ref{rmk:chi_r,m}. 
\begin{proposition}\label{prop:homeomorphism_type}
    Suppose both $m$ and $\chi_{r,m}$ are even. Then the quotient $Q_{r,m}=X_{r,m}/\langle c\rangle$ is orientation-preservingly homeomorphic to \[
    (\chi_{r,m}-1)\mathbb{CP}^2 \# \left( 5\chi_{r,m}-1-\frac{K_{r,m}-\ell}{2}\right)\CPbar,
    \]
    where \[
    \ell=\begin{cases}
        1 & \text{if $r$ is odd},\\
        2 & \text{if $r$ is even}.
    \end{cases}
    \]
\end{proposition}
\begin{proof}
    The quotient map $p:X_{r,m}\to Q_{r,m}$ is a double cover branched along $B_{r,m}:=p(X_{r,m}^c)$. Since $X_{r,m}-X_{r,m}^c$ is simply-connected, we have $\pi_1(Q_{r,m}-B_{r,m})\cong \mathbb{Z}_2$, generated by a meridian of any component of $B_{r,m}$. This implies that $Q_{r,m}$ is simply-connected. 

    Next, $X_{r,m}^c$ is Lagrangian in $X_{r,m}$, being the fixed locus of an anti-symplectic involution. Thus, its normal Euler number in $X_{r,m}$ equals $-e(X_{r,m}^c)=-\ell$. It follows that the topological Euler characteristic and the signature of $Q_{r,m}$ are given by \[
    e(Q_{r,m})=\frac{e(X_{r,m})+\ell}{2}, \quad  \sigma(Q_{r,m})=\frac{\sigma(X_{r,m})-\ell}{2}.
    \]
    (For the signature, see \cite{Hirzebruch-1969}.) Therefore, the formulas for $e(X_{r,m})$ and $\sigma(X_{r,m})$ given in Proposition \ref{prop:properties_of_X_r,m}(b) yield \[
    b_+(Q_{r,m})=\chi_{r,m}-1, \quad b_-(Q_{r,m})=5\chi_{r,m}-1-\frac{K_{r,m}-\ell}{2}.
    \]

    Now, it remains to show that the intersection form of $Q_{r,m}$ is odd. Since both $m$ and $\chi_{r,m}$ are even, Remark \ref{rmk:chi_r,m} shows that one of the following conditions holds: \[
    \begin{cases}
        r\equiv0 \pmod{4}, \\
        r\equiv2 \pmod{4} \quad \text{and} \quad m\equiv 0\pmod{4}, \\
        r\equiv 3\pmod{4} \quad \text{and} \quad m\equiv 2\pmod{4}.
    \end{cases}
    \]
    When $r\equiv 0$ or $3\pmod{4}$, $\sigma(Q_{r,m})$ is odd, so $Q_{r,m}$ must have odd intersection form. 
    
    Next, consider the case $r=2$. Recall that \[
    X_{2,m}=\left\{(x,w)\in \mathcal{O}_{\mathbb{CP}^2}(m):w^{\otimes 2}=F(x)\in \mathcal{O}_{\mathbb{CP}^2}(2m)\right\},
    \]
    and that our involution $c$ covers complex conjugation $\kappa:\mathbb{CP}^2\to \mathbb{CP}^2$, i.e., $\pi\circ c=\kappa \circ \pi$, where $\pi:X_{2,m}\to \mathbb{CP}^2$ is the projection. The involution $\delta:X_{2,m}\to X_{2,m}$, $(x,w)\mapsto (x,-w)$, commutes with $c$, so it induces an involution $\bar{\delta}:Q_{2,m}\to Q_{2,m}$. Moreover, we have $Q_{2,m}/\langle \bar{\delta}\rangle =X_{2,m}/\langle c,\delta\rangle = \mathbb{CP}^2/\langle \kappa\rangle=S^4$, and the following diagram commutes: \[
    \begin{tikzcd}
        X_{2,m} \ar[d, "p"'] \ar[r, "\pi"] & \mathbb{CP}^2 \ar[d, "/\langle \kappa\rangle"] \\ Q_{2,m} \ar[r, "/\langle \bar{\delta}\rangle"] & S^4.
    \end{tikzcd}
    \] 
    The map $Q_{2,m}\to S^4$ is a double cover branched along $B/\langle \kappa\rangle\subset S^4$, where $B:=\{F=0\}\subset \mathbb{CP}^2$ is the branch locus of $\pi$. Since $F$ has real coefficients, complex conjugation $\kappa:B\to B$ is a free, anti-holomorphic involution. In particular, the quotient $B/\langle \kappa\rangle$ is a closed nonorientable surface. A special case of \cite[Theorem 1.1]{Nagami-2000} shows that a double cover of $S^4$ branched along a smoothly embedded, connected, closed surface $\Gamma$ is spin if and only if $\Gamma$ is orientable. We therefore conclude that $Q_{2,m}$ is non-spin; hence it has odd intersection form. 

    Finally, consider the case $r\equiv 2\pmod{4}$ with $r>2$, write $r=2s$ with $s$ odd. Then the map $\mathcal{O}(m)\to \mathcal{O}(sm)$, $v\mapsto v^{\otimes s}$, restricts to a map $X_{r,m}\to X_{2,sm}$ of degree $s$, where we use the same polynomial of degree $rm$ for both $X_{r,m}$ and $X_{2,sm}$. The induced map $f:Q_{r,m}\to Q_{2,sm}$ on quotients also has degree $s$. Since $Q_{2,sm}$ has odd intersection form, there exists $\alpha\in H^2(Q_{2,sm};\mathbb{Z})$ so that $\langle \alpha^2,[Q_{2,sm}]\rangle$ is odd. Then the pullback $f^*\alpha$ satisfies $\langle (f^*\alpha)^2,[Q_{r,m}]\rangle=s\langle \alpha^2,[Q_{2,sm}]\rangle$, which is again odd. Hence $Q_{r,m}$ has odd intersection form. This completes the proof.     
\end{proof}

\section{Complete intersection surfaces in \texorpdfstring{$\mathbb{CP}^4$}{CP4}}\label{sec:complete_intersection}

In this section, we prove Theorem \ref{thm:complete_intersection}. Let $p$ and $p'$ be homogeneous polynomials in five variables of positive degrees $d$ and $d'$, respectively. Suppose that their gradients are linearly independent at every nonzero common zero in $\mathbb{C}^5$. Then \[
S(d,d'):=\{p=p'=0\}\subset \mathbb{CP}^4
\]
is a smooth complex surface, called a \emph{complete intersection} surface of multidegree $(d,d')$. We will focus on the case $(d,d')=(2,m)$ with $m\geq 4$. The following proposition is well-known; see \cite[Section 1.3]{Gompf-Stipsicz-1999}, for example. 
\begin{proposition}\label{prop:properties_of_s(2,m)}
    For $m\geq 4$, the complete intersection $S_m:=S(2,m)$ has the following properties. \begin{enumerate}[label=\textup{(\alph*)}, leftmargin=2em]
    \item Its oriented diffeomorphism type depends only on $m$. Moreover, the diffeomorphisms can be chosen to identify the class $H:=c_1(\mathcal{O}_{\mathbb{CP}^4}(1)|_{S_m})$. Its canonical line bundle $K_{S_m}\cong \mathcal{O}_{\mathbb{CP}^4}(m-3)|_{S_m}$ is ample.
    \item It is simply-connected, and its topological Euler characteristic and the signature are given by \[
    e(S_m)=2m(m^2-3m+4) \quad \text{and} \quad \sigma(S_m)=-\frac{2m(m^2-1)}{3}.
    \]
    In particular, we have \[
    b_+(S_m)=\frac{m(2m^2-9m+13)}{3}-1.
    \]
    \item It is spin if and only if $m$ is odd.
    \item It is a minimal surface of general type.
    \end{enumerate}
\end{proposition}

From now on, we will use specific polynomials to define $S(2,m)$. Set \begin{align*}
    q&:=-z_0^2+z_1^2+z_2^2+z_3^2+z_4^2 \in H^0(\mathbb{CP}^4,\mathcal{O}_{\mathbb{CP}^4}(2)), \\
    C&:=\{q=0\}\subset \mathbb{CP}^4, \\
    P_m&:=z_1^{m-1}+2^{m-3}z_2^{m-1}+4^{m-3}z_3^{m-1} \in H^0(\mathbb{CP}^4,\mathcal{O}_{\mathbb{CP}^4}(m-1)), \\
    \lambda_m&:=8m(1+2^{m-3}+4^{m-3})\pi, \\
    F_m&:= z_4(\lambda_mz_0^{m-1}+z_4^{m-1})+(z_0-z_1)P_m \in H^0(\mathbb{CP}^4,\mathcal{O}_{\mathbb{CP}^4}(m)), \\
    X_m&: =C \cap \{F_m=0\}\subset \mathbb{CP}^4.
\end{align*}

\begin{proposition}\label{prop:involution_on_complete_intersection}
   For every $m\geq 4$, $X_m$ is a smooth complete intersection surface of multidegree $(2,m)$. Moreover, there exists a smooth involution $c:X_m\to X_m$ satisfying the following conditions. \begin{enumerate}[label=\textup{(\alph*)}, leftmargin=2em]
        \item The involution $c$ is anti-holomorphic and orientation-preserving. Its fixed locus $X_m^c$ is a smoothly embedded $2$-sphere.
        \item There is a K\"ahler form $\omega$ on $X_m$ such that $[\omega]=H$ and $c^*\omega=-\omega$. 
        \item The complement $X_m-X_m^c$ is simply-connected.
    \end{enumerate}
\end{proposition}

To prove this proposition, we shall use two elementary lemmas, one algebraic and one geometric.
\begin{lemma}\label{lem:transcendental}
Let $E\subset \mathbb{C}$ be a finitely generated field extension of $\mathbb{Q}$. Suppose $\lambda \in E$ is transcendental over $\mathbb{Q}$. Then there exists a $\mathbb{Q}$-derivation $D:E\to E$ such that $D(\lambda)=1$.
\end{lemma}
Recall that a $\mathbb{Q}$-linear map $E\to E$ is called a \emph{$\mathbb{Q}$-derivation} if it satisfies Leibniz's rule. 
\begin{proof}
    This follows by extending $\lambda$ to a transcendence basis of $E$ over $\mathbb{Q}$, and then extending $\frac{\partial}{\partial \lambda}$ to the resulting finite separable extension. 
\end{proof}

\begin{lemma}\label{lem:geometric_dual}
    Let $X$ be a simply-connected smooth $4$-manifold, and let $F\subset X$ be a smoothly embedded, connected, closed surface. If there exists a smoothly embedded $2$-sphere $S\subset X$ intersecting $F$ transversely at a single point, then $X-F$ is simply-connected.
\end{lemma}
\begin{proof}  Removing a small open disk around the unique intersection point $S\cap F$ from $S$ gives a disk in $X-F$ whose boundary is a meridian of $F$. This proves the result since $\pi_1(X-F)$ is normally generated by a meridian of $F$. 
\end{proof}

\begin{proof}[Proof of Proposition \ref{prop:involution_on_complete_intersection}]
    We first show that $X_m$ is smooth. Suppose that there exists a nonzero $z=(z_0,\dots,z_4)\in \mathbb{C}^5$ such that $q(z)=F_m(z)=0$, and the gradients $\nabla F_m(z)$ and $\nabla q(z)$ are linearly dependent. Since $\nabla q(z)\neq 0$, there exists $\mu\in \mathbb{C}$ such that $\nabla F_m(z)=\mu \nabla q(z)$. Since $\lambda_m$ is transcendental over $\mathbb{Q}$, the field $E:=\mathbb{Q}(\lambda_m,z_0,\dots,z_4,\mu)$ admits a $\mathbb{Q}$-derivation $D$ satisfying $D(\lambda_m)=1$ by Lemma \ref{lem:transcendental}. Applying $D$ to the equations $q(z)=0$ and $F_m(z)=0$ gives \[
    \sum_{j=0}^4 \frac{\partial q}{\partial z_j}(z) D(z_j)=0=z_4z_0^{m-1}+\sum_{j=0}^4 \frac{\partial F_m}{\partial z_j}(z)D(z_j).
    \]
    Since $\nabla F_m(z)=\mu \nabla q(z)$, it follows that $z_4z_0^{m-1}=0$. If $z_0\neq 0$, then $z_4=0$, and $\frac{\partial F_m}{\partial z_4}(z)=\lambda_m z_0^{m-1}\neq 0$ and $\frac{\partial q}{\partial z_4}(z)=0$, a contradiction. Thus, $z_0=0$, so $0=\mu \frac{\partial q}{\partial z_0}(z)=\frac{\partial F_m}{\partial z_0}(z)=P_m(z)$. Since $F_m(z)=0$, it follows that $z_4=0$. Consequently, \[
    0=q(z)=z_1^2+z_2^2+z_3^2, \quad 0=P_m(z)=z_1^{m-1}+2^{m-3}z_2^{m-1}+4^{m-3}z_3^{m-1}.
    \]
    Hence $z_1\neq 0$; otherwise the first equation gives $|z_2|=|z_3|\neq 0$, so the second would give $2^{m-3}=4^{m-3}$, a contradiction. Now $\frac{\partial F_m}{\partial z_1}(z)=\mu \frac{\partial q}{\partial z_1}(z)$ gives $2\mu = -(m-1)z_1^{m-2}$. Then the equations $\frac{\partial F_m}{\partial z_2}(z)=\mu \frac{\partial q}{\partial z_2}(z)$ and $\frac{\partial F_m}{\partial z_3}(z)=\mu \frac{\partial q}{\partial z_3}(z)$ give $2^{m-3}z_2^{m-2}=z_1^{m-3}z_2$ and $4^{m-3}z_3^{m-2}=z_1^{m-3}z_3$, which imply $|z_2|\leq |z_1|/2$ and $|z_3|\leq |z_1|/4$. But this yields a contradiction: \[
    |z_1|^2 = \left|z_2^2+z_3^2\right|\leq |z_2|^2+|z_3|^2 \leq \frac{5}{16}|z_1|^2.
    \]
    This completes the proof that $X_m$ is a smooth complete intersection surface of multidegree $(2,m)$. 

    Since both $q$ and $F_m$ have real coefficients, complex conjugation on $\mathbb{CP}^4$ restricts to an anti-holomorphic involution $c$ on $X_m$. It preserves the complex orientation since $\dim_{\mathbb{C}}X_m=2$. If $\omega_{\mathrm{FS}}$ is the Fubini--Study form on $\mathbb{CP}^4$ normalized so that $[\omega_{\mathrm{FS}}]=c_1(\mathcal{O}_{\mathbb{CP}^4}(1))$, then its restriction $\omega:=\omega_{\mathrm{FS}}|_{X_m}$ to $X_m$ is a K\"ahler form with $[\omega]=H$ and $c^*\omega=-\omega$, proving part (b).

    The fixed locus $X_m^c$ is precisely the real locus of $X_m$, i.e., $X_m^c=X_m\cap \mathbb{RP}^4$. Observe that every real point of $C=\{q=0\}$ has $z_0\neq 0$. Thus $C_{\mathbb{R}}:=C\cap \mathbb{RP}^4$ is contained in the affine chart $z_0=1$, and is identified with $S^3$: \[
    C_{\mathbb{R}}=\{x_1^2+x_2^2+x_3^2+x_4^2=1\} \cong S^3.
    \]
    Under this identification, $X_m^c\subset S^3$ is the zero set of the function $g:S^3\to \mathbb{R}$ defined by \[
    g(x)=x_4+\frac{1}{\lambda_m}\left(x_4^m+(1-x_1)\left(x_1^{m-1}+2^{m-3}x_2^{m-1}+4^{m-3}x_3^{m-1} \right) \right)=:x_4+\frac{1}{\lambda_m} R(x).
    \]
    Let us show that $g^{-1}(0)\subset S^3$ is a smoothly embedded $2$-sphere. Letting $s_m:=1+2^{m-3}+4^{m-3}$, we have \[
    |R(x)| \leq 3s_m , \quad \|\nabla R(x)\| \leq 3ms_m \quad (x\in S^3),
    \]
    where the gradient $\nabla R$ is computed by regarding $R$ as a function on $\mathbb{R}^4$. In particular, if $g(x)=0$, then $|x_4|<\frac{1}{2}$ since $\lambda_m=8ms_m\pi >8ms_m$. Parametrize $S^3\cap \{|x_4|\leq \frac{1}{2}\}$ by the map $\Psi:S^2\times [-\frac{1}{2},\frac{1}{2}]\to S^3$ defined by $\Psi(y,t)=(\sqrt{1-t^2}y,t)$, and set $G(y,t)=g(\Psi(y,t))$. Since $G(y,\pm\frac{1}{2})=\pm\frac{1}{2}+\frac{1}{\lambda_m}R(\Psi(y,\pm \frac{1}{2}))$, we have $G(y,-\frac{1}{2})<0<G(y,\frac{1}{2})$. Moreover, we have \begin{align*}
        \left|\frac{\partial G}{\partial t}(y,t)-1 \right|=\left|\frac{1}{\lambda_m}\frac{\partial R(\Psi(y,t))}{\partial t} \right| \leq \frac{1}{\lambda_m}\ \left\|\nabla R(\Psi(y,t)) \right\|\cdot \left\|\frac{\partial \Psi}{\partial t}\right\| \leq \frac{3ms_m}{\lambda_m}\cdot \frac{2}{\sqrt{3}}<\frac{1}{2}.
    \end{align*}
    Thus, for each $y\in S^2$, there is a unique solution $t=t(y)\in (-\frac{1}{2},\frac{1}{2})$ to $G(y,t)=0$, and the map $y\mapsto t(y)$ is smooth by the implicit function theorem. Hence the map $y\mapsto \Psi(y,t(y))$ is a diffeomorphism of $S^2$ onto $X_m^c$, proving part (a).

    Finally, consider the line \[
    L:=\left\{ z_4=0, \, z_0=z_1,\,z_2=iz_3 \right\}\subset \mathbb{CP}^4,
    \]
    which is contained in $X_m$ and meets $X_m^c$ exactly at the point $p:=[1:1:0:0:0]$. Let us show that this intersection is transverse. Use affine coordinates $u_j=z_j/z_0$  ($j=1,\dots,4$) on $U:=\{z_0\neq 0\}\subset\mathbb{CP}^4$. Define $q_0(u)=q(1,u_1,\dots,u_4)$ and $f_m(u)=F_m(1,u_1,\dots,u_4)$; then $X_m\cap U=\{q_0=f_m=0\}$. Computing the differentials of $q_0,f_m:\mathbb{C}^4\to \mathbb{C}$ shows that $T_pX_m=\{v\in \mathbb{C}^4:v_1=v_4=0\}$, and therefore $T_pX_m^c=\{v\in \mathbb{R}^4:v_1=v_4=0\}$. On the other hand, we have $T_pL=\{v\in \mathbb{C}^4:v_1=v_4=0,\,v_2=iv_3\}$. Thus, $T_pX_m^c\cap T_pL=\{0\}$, which shows that $X_m^c$ and $L$ intersect transversely at $p$. Now, Lemma \ref{lem:geometric_dual} implies that $X_m-X_m^c$ is simply-connected, proving part (c).    
\end{proof}

Now we can prove Theorem \ref{thm:complete_intersection}.

\begin{proof}[Proof of Theorem \ref{thm:complete_intersection}]
     For each $k\geq 1$, let $Y_k:=X_{4k}$. By Proposition \ref{prop:involution_on_complete_intersection}, it is a smooth complete intersection surface of multidegree $(2,4k)$, and there is an anti-holomorphic involution $c$ on $Y_k$ such that the fixed locus $Y_k^c$ is a $2$-sphere and $Y_k-Y_k^c$ is simply-connected. Moreover, there is a K\"ahler form $\omega$ on $Y_k$ with $c^*\omega=-\omega$ and $[\omega]=H$. Now, Proposition \ref{prop:properties_of_s(2,m)} shows that $Y_k$ is a non-spin complex surface of general type, and that \[
     b_+(Y_k)=\frac{4k(32k^2-36k+13)}{3}-1.
     \]
     It is easy to see that $b_+(Y_k)\equiv 3\pmod{4}$. Consequently, $(Y_k,\omega,c)$ satisfies the hypotheses of Corollary \ref{cor:symplectic}, so Theorem \ref{thm:main} implies $Y_k$ admits an infinite exotic family $\{c_n\}_{n=0}^\infty$ of involutions with $c_0=c$, such that every $c_n$ with $n\geq 1$ is not anti-symplectic with respect to any orientation-compatible symplectic form. The assertion about exotic surfaces in $Y_k/\langle c\rangle$ follows from Theorem \ref{thm:isotopy}. Proposition \ref{prop:properties_of_s(2,m)}(b) gives simple connectivity of $Y_k$, and irreducibility of $Y_k$ can be proved exactly as in the proof of Theorem \ref{thm:branched_cover}. \end{proof}

\begin{remark}\label{rmk:complete_antisymplectic}
    Although $c_n$ for each $n\geq 1$ is not anti-symplectic with respect to any orientation-compatible symplectic form on $Y_k$, it satisfies $c_n^*[\omega]=-[\omega]$ on cohomology for the K\"ahler form $\omega$ used in the proof of Theorem \ref{thm:complete_intersection}. This can be seen exactly as in Remark \ref{rmk:cohomology_antisymplectic}, noting that $c_1(K_{Y_k})=(4k-3)H$ by Proposition \ref{prop:properties_of_s(2,m)}(a).
\end{remark}

We end this section by determining the oriented homeomorphism type of the quotient for the involution $c$ defined in Proposition \ref{prop:involution_on_complete_intersection}.
\begin{proposition}\label{prop:quotient_homeo_type_complete_intersection}
    For every $m\geq 4$, the quotient $Q_m:=X_m/\langle c\rangle$ is orientation-preservingly homeomorphic to \[
    \frac{(m-1)(m-2)(2m-3)}{6}\mathbb{CP}^2\#\frac{m(4m^2-9m+11)}{6}\CPbar.
    \]
\end{proposition}
\begin{proof}
    The argument is similar to the proof of Proposition \ref{prop:homeomorphism_type}, but more simple. $Q_m$ is simply-connected, and its topological Euler characteristic and signature are given by \[
    e(Q_m)=\frac{e(X_m)+2}{2}, \quad \sigma(Q_m)=\frac{\sigma(X_m)-2}{2},
    \]
    so the formulas in Proposition \ref{prop:properties_of_s(2,m)}(b) give \[
    b_+(Q_m)= \frac{(m-1)(m-2)(2m-3)}{6}, \quad b_-(Q_m)=\frac{m(4m^2-9m+11)}{6}.
    \]
    Note that $b_2(Q_m)=m(m^2-3m+4)-1$ is always odd, which implies that the intersection form of $Q_m$ is odd. This gives the result. 
\end{proof}

\bibliography{references}{}
\bibliographystyle{alpha}
\end{document}